\documentclass[12pt,a4paper]{article}
\usepackage{amsfonts,amsthm,amsmath,amssymb,amscd}
\usepackage{fancybox}
\usepackage[latin1]{inputenc}
\usepackage[usenames]{color}
\usepackage{graphicx}
\usepackage{enumerate}
\usepackage{float}
\usepackage{fullpage}
\usepackage{epsfig,amssymb} 
\usepackage{subcaption}

\newtheorem{theorem}{Theorem}[section]
\newtheorem{corollary}[theorem]{Corollary}
\newtheorem{lemma}[theorem]{Lemma}
\newtheorem{proposition}[theorem]{Proposition}

\newtheorem{rem}[theorem]{Remark}

\begin{document}
	
\title{\textbf{The Renormalized Area--Willmore Relation}}

\author{\'Alvaro P\'ampano and Aaron J. Tyrrell}
\date{\today}

\maketitle

\begin{figure}[H]
	\begin{center}
		\includegraphics[height=3.5cm]{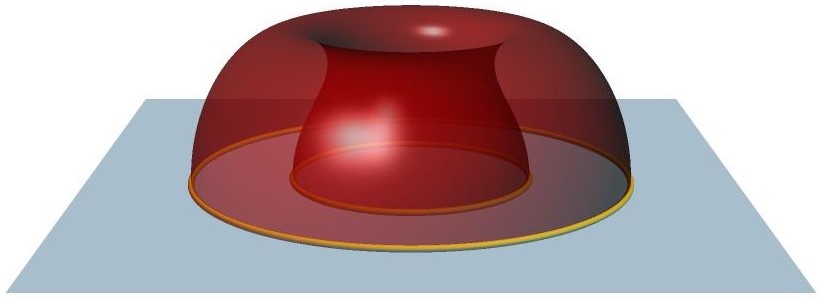}
	\end{center} 
\end{figure}

\begin{abstract} In \cite[Proposition 8.1]{AM}, Alexakis and Mazzeo showed that the renormalized area of minimal surfaces in the hyperbolic space $\mathbb{H}^3$ approaching the ideal boundary $\partial_\infty\mathbb{H}^3$ is the negative of the Willmore energy of their Euclidean conformal compactifications.
	
In this paper, we investigate whether there are other geometric surface equations for which these two quantities satisfy a universal proportionality relation. We answer this question in the affirmative by extending the Alexakis--Mazzeo result to a broader class of geometric equations containing, as two distinguished cases, the equations defining generalized singular minimal surfaces in $\mathbb{R}^3$ and linear Weingarten surfaces in $\mathbb{H}^3$.

Moreover, for rotational linear Weingarten surfaces, we express both energies in terms of generalized hyperelliptic integrals and show that their doubles form a family of rotational tori in $\mathbb{R}^3$ whose Willmore energies realize all values in $(2\pi^2,\infty)$, with a particular subfamily converging to the Clifford torus in $\mathcal{W}^{2,p}$ for every $1\leq p<\infty$, as the optimal value $2\pi^2$ is approached.\\

\end{abstract}

\section{Introduction}

The area of a properly embedded submanifold in the hyperbolic space diverges if the submanifold approaches the ideal boundary. Under suitable assumptions on its asymptotic behavior near the boundary, however, this divergence can be removed through Hadamard regularization. When the dimension of the submanifold is even, the constant term in the resulting asymptotic expansion defines a global invariant known as the renormalized area $\mathcal A_R$. This invariant originated in the study of Poincar\'e--Einstein spaces and is closely connected to the AdS/CFT correspondence in physics \cite{Gr,GW}.

The specific case of surfaces $M$ in $\mathbb{H}^3$ approaching the ideal boundary $\partial_\infty\mathbb{H}^3$ at a right angle was further investigated in the seminal paper \cite{AM}, where the Gauss--Bonnet-type formula
$$\mathcal{A}_R(M)=-2\pi\chi(M)-\frac{1}{2}\int_M \lvert\mathring{B}\rvert^2\,dA+\int_M H^2\,dA\,,$$
was established. This formula decomposes the renormalized area into three parts: a topological term determined by the Euler characteristic $\chi(M)$, a conformally invariant term involving the trace-free second fundamental form $\mathring B$, and an extrinsic geometric term corresponding to the hyperbolic bending energy, which is the integral of the square of the mean curvature $H$. Throughout this paper we follow the convention that $H$ is half the trace of the second fundamental form.

The upper half-space model of $\mathbb{H}^3$ provides a natural setting for translating the above hyperbolic problem into a Euclidean one. Indeed, under the standard conformal change of metric, a surface $M$ extending smoothly to $\partial_\infty\mathbb{H}^3=\{z=0\}$ may be viewed as a compact surface in the Euclidean upper half-space $\mathbb{R}^3_+=\mathbb{R}^3\cap\{z\geq 0\}$ whose boundary lies in the plane $\{z=0\}$. In particular, the condition that $M$ approaches $\partial_\infty\mathbb{H}^3$ at a right angle is equivalent to requiring that the corresponding Euclidean surface meets the plane $\{z=0\}$ orthogonally.

For such a surface, the Euclidean bending energy is
$$\mathcal{W}(M)=\int_M \overline{H}^2\,d\overline{A}\,,$$
where $\overline{H}$ denotes the Euclidean mean curvature of $M$. This energy differs from the conformally invariant Willmore functional \cite{Wi} by the total Gaussian curvature. For closed surfaces, the Gauss--Bonnet theorem shows that this difference is the topological term $2\pi\chi(M)$ and, in particular, both functionals agree on tori. Throughout the paper, $\mathcal{W}$ will simply be referred to as the Willmore energy.

A fundamental relation between the renormalized area $\mathcal{A}_R$ of minimal surfaces in $\mathbb H^3$ approaching $\partial_\infty\mathbb{H}^3$ and their Willmore energy $\mathcal{W}$, viewed as Euclidean surfaces in $\mathbb{R}^3_+$, was established by S. Alexakis and R. Mazzeo in \cite[Proposition 8.1]{AM}. They proved that $\mathcal{A}_R(M)=-\mathcal{W}(M)$ and noted that no analogous formula of comparable simplicity seems to exist for general non-minimal surfaces. The extent to which this phenomenon was particular to minimal surfaces remained unclear.

A decisive step toward understanding this phenomenon was obtained by B. Palmer and the first author in \cite[Theorem 3.1]{PP} (see also \cite[Theorem 4.2]{P} for a different derivation of the same identity). They showed that for every surface $M$ in $\mathbb{H}^3$ approaching $\partial_\infty\mathbb{H}^3$ at a right angle,
$$\mathcal{A}_R(M)=\int_M H^2\,dA-\int_M \overline{H}^2\,d\overline{A}\,.$$
In other words, the renormalized area is the difference between the hyperbolic and Euclidean bending energies.

The purpose of this paper is to determine whether the Alexakis--Mazzeo relation persists beyond the minimal setting. More precisely, the goal is to identify geometric surface equations for which there exists a constant $\lambda\in\mathbb R$, depending only on the equation, such that every solution satisfying the required asymptotic conditions obeys
$$\mathcal{A}_R(M)=\lambda\mathcal{W}(M)\,.$$
In view of the identity in \cite[Theorem 3.1]{PP}, the existence of such a universal proportionality relation is equivalent to the existence of a universal linear relation between the hyperbolic and Euclidean bending energies.

The main theorem shows that the relation between renormalized area and bending energies described above naturally singles out a broad class of geometric equations for which universal renormalized area--Willmore relations occur. Let $\kappa_1$ and $\kappa_2$ be the principal curvatures of $M$ in $\mathbb{H}^3$, $H=(\kappa_1+\kappa_2)/2$ denote the hyperbolic mean curvature of $M$, $R=\kappa_1\kappa_2$ its extrinsic scalar curvature, and $\nu$ the cosine of the hyperbolic angle between the unit normal to $M$ and the unit normal to the horizontal horospheres $\{z=z_o\in\mathbb{R}^+\}$. In Theorem \ref{main} we prove that every surface with $b\chi(M)=0$ satisfying a non-trivial equation of the form
$$a\,H^2+b\,R+c\left(H-\nu\right)^2=0\,,$$
necessarily has $a+b\neq 0$ and satisfies the universal relation
$$\mathcal{A}_R(M)=-\frac{a+c}{a+b}\,\mathcal{W}(M)\,.$$
When $b\neq 0$, the annular condition $\chi(M)=0$ arises solely from the Euler characteristic term in the Gauss--Bonnet type formula for $\mathcal{A}_R$, rather than from a local feature of the surface equation.

The proportionality relation is accompanied by a sharp rigidity phenomenon. Proposition \ref{p1} below shows that every identity $\mathcal{A}_R(M) = \lambda\mathcal{W}(M)$ must satisfy $\lambda\geq -1$, with equality if and only if $M$ is minimal in $\mathbb H^3$. Thus the Alexakis--Mazzeo identity is uniquely characterized by hyperbolic minimality. Conversely, every value $\lambda>-1$ arises from a one-parameter family of equations within the class identified in Theorem \ref{main}. To the best of our knowledge, these are the first non-minimal geometric surface equations known to impose a universal renormalized area--Willmore relation on all of their solutions meeting the ideal boundary at a right angle.

One distinguished example is the generalized singular minimal surface equation in $\mathbb R^3_+$, whose solutions are critical points of suitable weighted area functionals. After passage to the hyperbolic metric, this equation falls within the general framework above, yielding a universal renormalized area--Willmore relation for all solutions (see Corollary \ref{t2}). 

A second example is the linear Weingarten equation $\kappa_2=\mu\kappa_1$. Linear Weingarten annuli $M$ in $\mathbb{H}^3$ likewise satisfy a universal renormalized area--Willmore relation (see Corollary \ref{t3}). Together, these examples show that the proportionality principle arises naturally in geometrically unrelated settings.

Rotational non-totally geodesic linear Weingarten surfaces with $\mu<0$ are studied in greater detail. It was shown in \cite{Pa} that their generating curves are critical points of the $p$-elastic functional 
$$\mathbf{\Theta}_p(\gamma)=\int_\gamma \kappa^p\,ds, $$ 
with $p=1/(1-\mu)$. As a consequence of the geometric description of such curves established in Appendix A, these surfaces are properly embedded annuli approaching the ideal boundary at a right angle, and therefore have well defined renormalized area. Theorem \ref{prop} shows that both $\mathcal{A}_R$ and $\mathcal{W}$ are determined by the shifted functional $\mathbf{\Theta}_{p+1}$, providing a direct analogue of the classical relation between rotational Willmore tori and elastic curves established by J. Langer and D. Singer \cite{LS}. The $p$-elastic characterization further allows $\mathbf{\Theta}_{p+1}$, and hence both surface energies, to be expressed in terms of generalized hyperelliptic integrals, leading to a complete determination of their ranges (see Theorems \ref{t5} and \ref{cor}, respectively).

As an application, reflecting the linear Weingarten annuli across the plane $\{z=0\}$ produces a geometrically distinguished family of $\mathcal{C}^2$ rotational tori in $\mathbb R^3$ whose Willmore energies fill the interval $(2\pi^2,\infty)$. Moreover, the approach to the optimal value $2\pi^2$ is accompanied by global $\mathcal{W}^{2,p}$, for every $1\leq p<\infty$, convergence to the Clifford torus, with smooth convergence locally away from the reflection plane (see Theorem \ref{propnew}).

\section{Renormalized Area and Willmore Energy}

Consider the upper half-space model for the hyperbolic space $\mathbb{H}^3$ given by
$$\mathbb{H}^3=\{(x,y,z)\in\mathbb{R}^3\,\lvert\,z>0\}\,,$$
endowed with the metric
$$g=\frac{1}{z^2}\left(dx^2+dy^2+dz^2\right).$$

Let $M$ be a smooth connected surface properly embedded in $\mathbb{H}^3$, approaching the ideal boundary $\partial_\infty\mathbb{H}^3=\{z=0\}$ at a right angle, and which extends smoothly to a compact surface with boundary in the plane $\{z=0\}$. Unless otherwise stated, throughout this paper we shall assume that the surfaces under consideration satisfy the above hypotheses. Such a surface has infinite area but its Hadamard regularization is a global invariant of $M$. This invariant is defined as the constant term in the expansion of the area of truncated surfaces $M_\epsilon=M\cap\{z\geq \epsilon\}$ in terms of $\epsilon>0$ and is often referred to as the renormalized area \cite{Gr,GW}.

The renormalized area $\mathcal{A}_R$ of a surface $M$ in $\mathbb{H}^3$ approaching the ideal boundary $\partial_\infty\mathbb{H}^3$ at a right angle has the explicit Gauss--Bonnet type formula (first given in Proposition 3.1 of \cite{AM}; see also Corollary 4.4 of \cite{P})
\begin{equation}\label{ARgaussbonnet}
	\mathcal{A}_R(M)=-2\pi\chi(M)-\frac{1}{2}\int_M \lvert\mathring{B}\rvert^2\,dA+\int_M H^2\,dA\,,
\end{equation}
where $\chi(M)$ is the Euler characteristic of $M$, $\lvert\mathring{B}\rvert^2$ is the squared norm of the trace-free second fundamental form of $M$ in $\mathbb{H}^3$, $H$ is the mean curvature of $M$ (defined as half the trace of the second fundamental form $B$), and $dA$ represents the area element with respect to the induced (hyperbolic) metric on $M$.

Let $R$ denote the extrinsic scalar curvature as defined in \cite{C} (see also (1) of \cite{P}), that is,
\begin{equation}\label{R}
	R=\kappa_1\kappa_2\,,
\end{equation}
where $\kappa_1$ and $\kappa_2$ are the principal curvatures of $M$ (which are, in turn, the eigenvalues of the second fundamental form $B$). Then, from the standard relation (see Remark 3.1 of \cite{P} for details)
$$R=H^2-\frac{1}{2}\lvert\mathring{B}\rvert^2\,,$$
we can rewrite \eqref{ARgaussbonnet} simply as
\begin{equation}\label{ARhyperbolic}
	\mathcal{A}_R(M)=-2\pi\chi(M)+\int_M R\,dA\,.
\end{equation}

After the conformal change of metric $\overline{g}=z^2g$, the surface $M$ can be regarded as a smooth compact surface embedded in $\mathbb{R}^3_+=\mathbb{R}^3\cap\{z\geq 0\}$ whose boundary $\partial M$ lies in the plane $\{z=0\}$. The assumption that $M$ approaches $\partial_\infty\mathbb{H}^3$ at a right angle means that $M$ meets this plane orthogonally.

Let $\overline{H}$ denote the Euclidean mean curvature of $M$ and let $\overline{\xi}$ be the Euclidean unit normal vector field to $M$. The hyperbolic mean curvature $H$ is related to $\overline{H}$ by the identity
\begin{equation}\label{HH}
	H=z\overline{H}+\overline{\xi}_z\,,
\end{equation}
where $\overline{\xi}_z=\overline{g}(\overline{\xi},E_3)$ is the vertical component of $\overline{\xi}$ and $E_3$ denotes the Euclidean unit vector field in the direction of the $z$-axis. 

The Euclidean quantity $\overline{\xi}_z$ also has a natural hyperbolic interpretation. Indeed, the hyperbolic unit normal to $M$ is $\xi=z\overline{\xi}$, while $zE_3$ is the hyperbolic unit normal to the horizontal horospheres $\{z=z_o\in\mathbb{R}^+\}$. Consequently, 
$$g\left(\xi,zE_3\right)=g\left(z\overline{\xi},zE_3\right)=\overline{g}\left(\overline{\xi},E_3\right)=\overline{\xi}_z\,.$$
Thus, $\overline{\xi}_z$ is the cosine of the hyperbolic angle between the hyperbolic unit normal to $M$ and the hyperbolic unit normal to the horizontal horospheres. When working with $M$ as a surface in $\mathbb{H}^3$, we will denote $\overline{\xi}_z$ simply by $\nu$. In terms of this notation, the conformal transformation law \eqref{HH} becomes
\begin{equation}\label{HHnew}
	H=z\,\overline{H}+\nu\,.
\end{equation}

We define the Euclidean bending energy of the surface $M$ by
$$\mathcal{W}(M)=\int_M \overline{H}^2\,d\overline{A}\,,$$
where $d\overline{A}$ denotes the area element induced on $M$ by the Euclidean metric $\overline{g}$. For simplicity, we refer to $\mathcal{W}(M)$ as the Willmore energy\footnote{The classical definition of the (conformal) Willmore energy \cite{Wi} differs from $\mathcal{W}(M)$ by a total Gaussian curvature term, which can be reduced to a topological term after applying the Gauss--Bonnet theorem.} of the surface $M$ in $\mathbb{R}^3$. Notice that the Willmore energy is strictly positive for the surfaces considered in this paper. Indeed, if $\mathcal{W}(M)=0$, then $\overline{H}$ vanishes identically and $M$, regarded as a surface in $\mathbb{R}^3_+$, is minimal. Since $M$ meets the plane $\{z=0\}$ orthogonally, reflection across this plane would produce a smooth closed minimal surface in $\mathbb{R}^3$, which is not possible.

From \eqref{HHnew} and the relation $d\overline{A}=z^2dA$ between the Euclidean and hyperbolic area elements, we get that the Willmore energy can also be expressed entirely in terms of hyperbolic quantities as
\begin{equation}\label{Wh}
	\mathcal{W}(M)=\int_M \overline{H}^2\,d\overline{A}=\int_M \left(H-\nu\right)^2\,dA\,.
\end{equation}

The renormalized area $\mathcal{A}_R$ admits an alternative expression involving the (Euclidean) Willmore energy $\mathcal{W}$. More precisely, it was shown in Theorem 3.1 of \cite{PP} (see also Theorem 4.2 of \cite{P} for a different proof) that the renormalized area $\mathcal{A}_R$ of $M$ satisfies
\begin{equation}\label{difference}
	\mathcal{A}_R(M)=\int_M H^2\,dA-\int_M\overline{H}^2\,d\overline{A}=\int_M H^2\,dA-\mathcal{W}(M)\,.
\end{equation}
Thus, the renormalized area of $M$ is the difference between its hyperbolic and Euclidean bending energies. 

When $M$ is a minimal surface in $\mathbb{H}^3$, $H=0$ holds identically and formula \eqref{difference} reduces to
\begin{equation}\label{-1}
	\mathcal{A}_R(M)=-\mathcal{W}(M)\,.
\end{equation}
This relation was first obtained by S. Alexakis and R. Mazzeo in Proposition 8.1 of \cite{AM}\footnote{Actually, Proposition 8.1 of \cite{AM} computes the Willmore energy of the doubled closed surface, hence, the coefficient $1/2$ in that statement.}. They also observed that such a simple relation may not be expected to persist for non-minimal surfaces.

The following result shows that, in fact, the relation \eqref{-1} characterizes minimal surfaces in $\mathbb{H}^3$ among those approaching $\partial_\infty\mathbb{H}^3$ at a right angle. However, it also suggests considering other equally simple relations consisting of changing the proportionality constant in \eqref{-1}.

\begin{proposition}\label{p1} Let $M$ be a surface in $\mathbb{H}^3$ approaching the ideal boundary $\partial_\infty\mathbb{H}^3$ at a right angle. For a fixed $\lambda\in\mathbb{R}$, the identity
	$$\mathcal{A}_R(M)=\lambda\mathcal{W}(M)$$ 
holds if and only if
	$$\int_M H^2\,dA=\left(\lambda+1\right)\mathcal{W}(M)\,.$$
Moreover, $\lambda\geq -1$ with equality if and only if $M$ is minimal in $\mathbb{H}^3$.
\end{proposition} 
\begin{proof} 
Assume $\mathcal{A}_R(M)=\lambda\mathcal{W}(M)$ holds. From \eqref{difference}, we deduce that
	$$(\lambda+1)\,\mathcal{W}(M)=\lambda\mathcal{W}(M)+\mathcal{W}(M)=\mathcal{A}_R(M)+\mathcal{W}(M)=\int_M H^2\,dA\geq 0\,.$$
Since $\mathcal{W}(M)>0$, it follows that $\lambda\geq -1$. Moreover, equality holds if and only if $H$ vanishes identically.

Conversely, assume
$$\int_M H^2\,dA=(\lambda+1)\mathcal{W}(M)\,,$$
holds. Using \eqref{difference} to substitute the hyperbolic bending energy above, we obtain
$$(\lambda+1)\mathcal{W}(M)=\int_M H^2\,dA=\mathcal{A}_R(M)+\mathcal{W}(M)\,,$$
finishing the proof.
\end{proof}

Proposition \ref{p1} completely characterizes the case $\lambda=-1$. For other values of $\lambda>-1$, if this proportionality constant is allowed to vary with each individual surface $M$, then since $\mathcal{W}(M)>0$ one can simply take $\lambda=\mathcal{A}_R(M)/\mathcal{W}(M)$ and $\mathcal{A}_R(M)=\lambda\mathcal{W}(M)$ will always hold. Therefore, the meaningful question is whether a geometric surface equation prescribes the same proportionality constant $\lambda>-1$ for all its solutions approaching $\partial_\infty\mathbb{H}^3$ at a right angle.

The formulas \eqref{ARhyperbolic}, \eqref{Wh}, and \eqref{difference} above suggest a natural class of such equations, namely, a pointwise linear relation among the three quantities $H^2$, $R$, and $(H-\nu)^2$.

In the following theorem we show that every such linear relation induces a universal proportionality relation between $\mathcal{A}_R$ and $\mathcal{W}$.

\begin{theorem}\label{main}
	Let $M$ be a topological annulus in $\mathbb{H}^3$ approaching the ideal boundary $\partial_\infty\mathbb{H}^3$ at a right angle. If
	\begin{equation}\label{cond}
		aH^2+bR+c\left(H-\nu\right)^2=0\,,
	\end{equation}
	holds for constants $a,b,c\in\mathbb{R}$, not all zero, then $a+b\neq 0$ and
	$$\mathcal{A}_R(M)=-\frac{a+c}{a+b}\,\mathcal{W}(M)\,.$$
	If $b=0$, the same conclusion holds without any restriction on the topology of the surface $M$.
\end{theorem} 
\begin{proof}
Assume that the surface $M$ satisfies \eqref{cond}. Integrating this equation over $M$ with respect to the hyperbolic area element and using \eqref{ARhyperbolic}, \eqref{Wh}, and \eqref{difference}, we obtain
\begin{eqnarray*}
	0&=&a\int_M H^2\,dA+b\int_M R\,dA+c\int_M \left(H-\nu\right)^2\,dA\\
	&=&a\left(\mathcal{A}_R(M)+\mathcal{W}(M)\right)+b\left(2\pi\chi(M)+\mathcal{A}_R(M)\right)+c\,\mathcal{W}(M)\\
	&=&(a+b)\mathcal{A}_R(M)+(a+c)\mathcal{W}(M)+2\pi b\chi(M)\,.
\end{eqnarray*}
If $M$ is a topological annulus, then $\chi(M)=0$. The last term also vanishes, regardless of the topology of $M$, if $b=0$. Consequently, in either case,
$$(a+b)\mathcal{A}_R(M)+(a+c)\mathcal{W}(M)=0\,.$$
If $a+b\neq 0$, the result follows.

It remains to show that $a+b\neq 0$. Suppose, to the contrary, that $a+b=0$. Since $\mathcal{W}(M)>0$, the above linear relation implies that $a+c=0$. If $b=0$, then $a=c=0$, contradicting the assumption that the coefficients are not all zero. Hence, $b=-a\neq 0$ and $c=-a$. After dividing equation \eqref{cond} by $a$, we obtain
$$H^2-R-\left(H-\nu\right)^2=0\,.$$
By the conformal transformation law (see (11) of \cite{P} and \cite[Theorem 1]{C}) and \eqref{HHnew},
$$0=H^2-R-\left(H-\nu\right)^2=z^2(\,\overline{H}^2-\overline{R}\,)-z^2\overline{H}^2=-z^2\overline{R}\,,$$
where $\overline{R}=\overline{K}$ is the Euclidean Gaussian curvature of $M$. Since $z>0$ in the interior of $M$, it follows that $\overline{K}=0$ there and, by continuity, also on $\partial M$. Moreover, since $M$ meets the plane $\{z=0\}$ orthogonally, reflection across this plane produces a closed embedded $\mathcal{C}^2$ surface in $\mathbb{R}^3$ whose Gaussian curvature vanishes identically. This is impossible, since every closed embedded $\mathcal{C}^2$ surface in $\mathbb{R}^3$ has an elliptic point. Therefore, $a+b\neq 0$.
\end{proof}

\begin{rem} When $b\chi(M)=0$, Theorem \ref{main} states that for every surface $M$ in $\mathbb{H}^3$ approaching $\partial_\infty\mathbb{H}^3$ at a right angle and satisfying \eqref{cond} the linear relation
\begin{equation}\label{new}
	(a+b)\mathcal{A}_R(M)+(a+c)\mathcal{W}(M)=0\,,
\end{equation}
must hold. Thus, the conclusion of this result can be interpreted as saying that the functionals $\mathcal{A}_R$ and $\mathcal{W}$ are linearly dependent on the moduli space of admissible surfaces.

Furthermore, the linear relation \eqref{new} can be translated, by means of \eqref{difference}, into a linear relation between the hyperbolic and Euclidean bending energies, namely,
$$(a+b)\int_M H^2\,dA-(b-c)\mathcal{W}(M)=0\,.$$
\end{rem}

If $b=c=0$ in \eqref{cond}, then $H$ vanishes identically and the surface $M$ is minimal in $\mathbb{H}^3$. For this particular case, Theorem \ref{main} recovers Proposition 8.1 of \cite{AM}.

For a prescribed proportionality constant $\lambda>-1$, the parameters in \eqref{cond} must satisfy
$$(\lambda+1)\,a+\lambda \,b+c=0\,.$$
After identifying equations whose coefficients differ by multiplication by a nonzero constant, the above condition determines a one-parameter family of geometric equations. Every solution surface satisfying the hypotheses of Theorem \ref{main} obeys the same relation $\mathcal{A}_R(M)=\lambda\mathcal{W}(M)$.

In the next sections, we identify two distinguished members of this family whose equations have independent geometric origins: generalized singular minimal surfaces in $\mathbb{R}^3_+$ and linear Weingarten surfaces in $\mathbb{H}^3$.

\section{Generalized Singular Minimal Surfaces}\label{SMS}

In this section we consider surfaces arising as critical points of weighted area functionals in the Euclidean upper half-space \cite{G,F,F2}. More precisely, for a constant $\eta\in\mathbb{R}$, consider
$$\mathcal{F}_\eta (M)=\int_M z^{2\eta}\,d\overline{A}\,,$$
where $M$ is regarded as a surface in $\mathbb{R}^3_+=\mathbb{R}^3\cap\{z\geq 0\}$. The corresponding Euler--Lagrange equation is
\begin{equation}\label{ELsurfaces}
	\overline{H}=\eta\,\frac{\overline{\xi}_z}{z}\,.
\end{equation}
Surfaces satisfying \eqref{ELsurfaces} will be called generalized singular minimal surfaces, extending the terminology used for the classical case $\eta=1/2$ \cite{D}.

Equation \eqref{ELsurfaces} admits a simple interpretation in terms of hyperbolic quantities. Indeed, from \eqref{HHnew} and the notation $\nu=\overline{\xi}_z$, we obtain
$$H=z\,\overline{H}+\nu=\left(\eta+1\right)\nu\,.$$
This observation allows us to conclude that the equation defining generalized singular minimal surfaces in $\mathbb{R}^3$ belongs to the family \eqref{cond} with $b=0$, and therefore prescribes a universal relation between renormalized area $\mathcal{A}_R$ and Willmore energy $\mathcal{W}$.

\begin{corollary}\label{t2}
	Let $M$ be a surface in $\mathbb{H}^3$ approaching the ideal boundary $\partial_\infty\mathbb{H}^3$ at a right angle. If $H=(\eta+1)\nu$ for a constant $\eta\neq 0$, then
	$$\mathcal{A}_R(M)=\frac{2\eta+1}{\eta^2}\,\mathcal{W}(M)\,.$$
	(No restriction on the topology of $M$ is required.)
\end{corollary}
\begin{proof} 
	From the hypothesis $H=(\eta+1)\nu$, we have
	$$\eta^2H^2-(\eta+1)^2\left(H-\nu\right)^2=0\,,$$
	which is the particular case of equation \eqref{cond} given by the parameters
	$$a=\eta^2\neq 0\,,\quad\quad\quad b=0\,,\quad\quad\quad c=-(\eta+1)^2\,.$$
	Since $\eta\neq 0$, these coefficients are not all zero. Therefore, Theorem \ref{main} applies without any topological restriction (because $b=0$) and yields
	$$\mathcal{A}_R(M)=-\frac{\eta^2-(\eta+1)^2}{\eta^2}\,\mathcal{W}(M)=\frac{2\eta+1}{\eta^2}\,\mathcal{W}(M)\,,$$
	as claimed.
\end{proof}

This proportionality relation is not vacuous. Generalized singular minimal surfaces invariant under translations and rotations were classified in \cite{L}, and rotationally symmetric solutions were further studied in \cite{DG,KL}. In particular, the rotational families described in those works contain surfaces that meet the plane $\{z=0\}$ orthogonally, and hence provide non-trivial examples to which Corollary \ref{t2} applies.

\section{Linear Weingarten Annuli}\label{LWA}

Let $M$ be a surface properly embedded in $\mathbb{H}^3$, and let $\kappa_1$ and $\kappa_2$ denote its principal curvatures. A surface whose principal curvatures satisfy a nontrivial functional relation is called a Weingarten surface, following the classical work of Weingarten \cite{W}. In this paper, we consider the (homogeneous) linear Weingarten relation
\begin{equation}\label{LW}
	\kappa_2=\mu\,\kappa_1\,,
\end{equation}
for some constant $\mu\neq 0$. A surface $M$ satisfying \eqref{LW} is said to be a linear Weingarten surface. When $\mu=1$, the surface is totally umbilical in $\mathbb{H}^3$ and so either (a part of) a totally geodesic plane, an equidistant surface, a horosphere, or a geodesic sphere. Among these surfaces, only totally geodesic planes approach $\partial_\infty\mathbb{H}^3$ at a right angle. We therefore assume $\mu\neq 1$ throughout the remainder of the paper.

The relation \eqref{LW} gives $H=(\mu+1)\kappa_1/2$ and $R=\mu\kappa_1^2$. Consequently,
\begin{equation}\label{LWeq}
	4\mu H^2-(\mu+1)^2R=0\,.
\end{equation}
Thus, linear Weingarten surfaces in $\mathbb{H}^3$ provide a distinguished case of the family of equations \eqref{cond}.

\begin{corollary}\label{t3}
	Let $M$ be a topological annulus in $\mathbb{H}^3$ approaching the ideal boundary $\partial_\infty\mathbb{H}^3$ at a right angle. If $\kappa_2=\mu\kappa_1$ for some constant $\mu\neq 0,1$, then
	$$\mathcal{A}_R(M)=\frac{4\mu}{(\mu-1)^2}\,\mathcal{W}(M)\,.$$
\end{corollary}
\begin{proof} 
	Observe that \eqref{LWeq} is of the form \eqref{cond}, with $a=4\mu\neq 0$, $b=-(\mu+1)^2$, and $c=0$. The coefficients are not all zero, so Theorem \ref{main} applies. Since
	$$a+b=4\mu-(\mu+1)^2=-(\mu-1)^2\neq 0\,,$$
	we conclude that
	$$\mathcal{A}_R(M)=\frac{4\mu}{(\mu-1)^2}\,\mathcal{W}(M)\,,$$
	as claimed.
\end{proof}

To show that the class of surfaces considered in Corollary \ref{t3} contains nontrivial examples satisfying the required asymptotic behavior, we now restrict our attention to rotational linear Weingarten surfaces.

Let $M$ be invariant under the group of rotations about the $z$-axis. Then, $M$ is generated by a curve $\gamma\subset\mathbb{H}^2=\{(x,0,z)\in\mathbb{H}^3\}$, known as the generating curve, and the orbits of the rotation are Euclidean circles. We label the principal curvatures so that $\kappa_1=\kappa$ is the hyperbolic curvature of $\gamma$. If $\kappa$ is constant, then \eqref{LW} implies that $M$ is isoparametric. By the classification of isoparametric surfaces in $\mathbb{H}^3$ \cite{Ca}, such a surface is either totally umbilical or a spherical cylinder. However, approaching $\partial_\infty\mathbb{H}^3$ at a right angle forces the constant curvature $\kappa$ to vanish and so $M$ must be a totally geodesic plane. We henceforth assume that $\kappa$ is non-constant.

Rotational non-isoparametric linear Weingarten surfaces in $\mathbb{H}^3$ were characterized in Theorem 5.2.1 of \cite{Pa} (see also \cite{LP} for the analogous result in the Euclidean setting) in terms of $p$-elastic curves in $\mathbb{H}^2$. More precisely, a rotational non-isoparametric surface satisfies the relation \eqref{LW} for some $\mu\neq 0,1$ if and only if its generating curve satisfies the Euler--Lagrange equation associated to the $p$-elastic functional
$$\mathbf{\Theta}_p(\gamma)=\int_\gamma \kappa^p\,ds\,,$$
for $p=1/(1-\mu)$. Critical points in $\mathbb{H}^2$ for $p$-elastic functionals $\mathbf{\Theta}_p$ have been previously investigated (see, for instance, \cite{LS2,MP,Pa,PST}).

Since in our case $\kappa$ is assumed to be non-constant, the Euler--Lagrange equation of $\mathbf{\Theta}_p$ is equivalent to its first integral, which is the first-order ordinary differential equation
\begin{equation}\label{EL}
	\left(\kappa_s\right)^2=\frac{(1-\mu)^2\kappa^2}{\mu^2}\left(-\mu^2\kappa^2+c^2\kappa^{-2\mu/(1-\mu)}+1\right),
\end{equation}
where $c=\lvert 1-\mu\rvert\sqrt{a}$ and $a>0$ is a positive constant. The sign of the constant of integration $a$ determines the type of the orbits of the rotation; the case $a>0$ corresponds precisely to spherical rotations (for a detailed proof of this statement the reader is referred to Proposition 3.2.5 of \cite{Pa}).

The same characterization directly implies that $M$ can be viewed as the warped product surface $\gamma\times_f\mathbb{S}^1$, where the warping function is given by
\begin{equation}\label{f}
	f=\frac{1}{c}\,\kappa^{\,\mu/(1-\mu)}\,,
\end{equation}
for $c=\lvert 1-\mu\rvert \sqrt{a}$.

For fixed $\mu<0$ and $c>0$, equation \eqref{EL} determines a unique non-constant curvature $\kappa_{\mu,c}$ up to a translation of the arc length parameter. By the fundamental theorem of curves in $\mathbb{H}^2$, the function $\kappa_{\mu,c}$ determines a unique curve $\gamma_{\mu,c}$ up to an ambient isometry. We denote the corresponding rotational linear Weingarten surface by $M_{\mu,c}\cong\gamma_{\mu,c}\times_f\mathbb{S}^1$.

It remains to determine which of the surfaces $M_{\mu,c}$ approach the ideal boundary $\partial_\infty\mathbb{H}^3$ at a right angle. The following result shows that every such surface is a topological annulus approaching $\partial_\infty\mathbb{H}^3$ at a right angle.

\begin{theorem}\label{prop} For every $\mu<0$ and $c>0$, the rotational non-isoparametric linear Weingarten surface $M_{\mu,c}$ in $\mathbb{H}^3$ is a topological annulus and it approaches the ideal boundary $\partial_\infty\mathbb{H}^3$ at a right angle.
	
Moreover, the renormalized area $\mathcal{A}_R$ and the Willmore energy $\mathcal{W}$ of $M_{\mu,c}$ are given, respectively, by
$$\mathcal{A}_R(M_{\mu,c})=\frac{2\mu\pi}{c}\int_{\gamma_{\mu,c}}\kappa^{p+1}\,ds\,,\quad\quad\quad \mathcal{W}(M_{\mu,c})=\frac{(\mu-1)^2\pi}{2c}\int_{\gamma_{\mu,c}}\kappa^{p+1}\,ds\,,$$
for $p=1/(1-\mu)$.
\end{theorem}
\begin{proof} 
	As shown in Appendix A, the generating curve $\gamma_{\mu,c}$ is a complete embedded curve with two endpoints approaching the ideal boundary $\partial_\infty\mathbb{H}^2$ at a right angle. Moreover, $\gamma_{\mu,c}$ does not intersect the axis of rotation. Consequently, its corresponding rotational surface $M_{\mu,c}$ is a properly embedded topological annulus. The orthogonal approach of $\gamma_{\mu,c}$ to $\partial_\infty\mathbb{H}^2$ implies that $M_{\mu,c}$ approaches $\partial_\infty\mathbb{H}^3$ at a right angle.

	For the second assertion, we use \eqref{ARhyperbolic} to compute the renormalized area $\mathcal{A}_R$ of $M_{\mu,c}$. Since $M_{\mu,c}$ is a topological annulus, $\chi(M_{\mu,c})=0$, and hence
	$$\mathcal{A}_R(M_{\mu,c})=\int_{M_{\mu,c}} R\,dA=\mu\int_{M_{\mu,c}} \kappa^2\,dA=\frac{2\mu\pi}{c}\int_{\gamma_{\mu,c}} \kappa^{(2-\mu)/(1-\mu)}\,ds=\frac{2\mu\pi}{c}\int_{\gamma_{\mu,c}}\kappa^{p+1}\,ds\,.$$
	In the second equality, we used that $R=\mu\kappa_1^2=\mu\kappa^2$ for rotational linear Weingarten surfaces, while in the third we employed that $M_{\mu,c}\cong\gamma_{\mu,c}\times_f\mathbb{S}^1$, where $f$ is given in \eqref{f}, and so $dA=f d\theta ds$. In particular, since $\mathcal{A}_R(M_{\mu,c})$ is finite, the integral on the right-hand side converges.
	
	Finally, the expression of the Willmore energy $\mathcal{W}$ of $M_{\mu,c}$ follows from Corollary \ref{t3}.
\end{proof}

Theorem \ref{prop} reveals a shift in the curvature-dependent energy associated with the generating curve: although $\gamma_{\mu,c}$ is critical for the $p$-elastic functional $\mathbf{\Theta}_p$ for $p=1/(1-\mu)$, the renormalized area $\mathcal{A}_R$ and the Willmore energy $\mathcal{W}$ of the corresponding rotational surface $M_{\mu,c}$ are both determined by its $(p+1)$-elastic energy $\mathbf{\Theta}_{p+1}$, namely,
$$\mathbf{\Theta}_{p+1}(\gamma_{\mu,c})=\int_{\gamma_{\mu,c}}\kappa^{p+1}\,ds\,.$$

The fact that the Willmore energy $\mathcal{W}$ of the rotational linear Weingarten annuli $M_{\mu,c}$ is determined by a functional of the form $\mathbf{\Theta}_{p+1}$ is reminiscent of the classical reduction of Langer and Singer \cite{LS} for Euclidean rotational tori. In their setting, the Willmore energy of a rotational torus is, up to a universal multiplicative constant, the classical elastic energy $\mathbf{\Theta}_2$ of its generating curve, regarded as a curve in the hyperbolic plane $\{(x,0,z)\,\lvert\,x>0\}$ with metric $(dx^2+dz^2)/x^2$. There is, however, an important geometric distinction between the two constructions. In the present setting, $\mathbb{H}^2=\{(x,0,z)\,\lvert\,z>0\}$ is a totally geodesic plane of the ambient hyperbolic space and its induced metric is $(dx^2+dz^2)/z^2$. Thus, the two generating curves are regarded as curves in different hyperbolic planes with conformal factors $x^{-2}$ and $z^{-2}$, respectively. In addition, in the Langer--Singer setting the ideal boundary $\{x=0\}$ is the axis of rotation, whereas in the present setting the axis of rotation is a geodesic contained in the interior of $\mathbb{H}^2$.

In the following section, we use the first integral \eqref{EL} to evaluate this $(p+1)$-elastic energy $\mathbf{\Theta}_{p+1}$ in terms of generalized hyperelliptic integrals.

\section{Renormalized Area and Willmore Energy of Rotational Linear Weingarten Annuli}\label{5}

In this section we evaluate the $(p+1)$-elastic energy $\mathbf{\Theta}_{p+1}$ appearing in Theorem \ref{prop} by means of the first integral \eqref{EL}. This yields explicit formulas for the renormalized area $\mathcal{A}_R$ and Willmore energy $\mathcal{W}$ of the rotational linear Weingarten annuli $M_{\mu,c}$, as well as the ranges of both functionals in the family $\{M_{\mu,c}\,\lvert\,\mu<0,c>0\}$.

\begin{theorem}\label{t5}
	Let $\mu<0$ and $c>0$. The renormalized area $\mathcal{A}_R$ of the rotational linear Weingarten annulus $M_{\mu,c}$ is given by the convergent improper integral
	\begin{equation}\label{ARhyperelliptic}
		\mathcal{A}_R(M_{\mu,c})=-\frac{4\mu^2\pi}{c}\int_0^\alpha \frac{x^{1-\mu}}{\sqrt{Q_{\mu,c}(x)}}\,dx\,,
	\end{equation}
	where $\alpha>0$ is the only positive solution of $Q_{\mu,c}(x)=-\mu^2x^{2(1-\mu)}+c^2x^{-2\mu}+1=0$.
	
	Moreover, as $\mu<0$ and $c>0$ vary, the renormalized area $\mathcal{A}_R(M_{\mu,c})$ takes every value in $(-\infty,0)$.
\end{theorem}
\begin{proof} 
	By Theorem \ref{prop}, the renormalized area $\mathcal{A}_R$ of $M_{\mu,c}$ is given by
	$$\mathcal{A}_R(M_{\mu,c})=\frac{2\mu\pi}{c}\int_{\gamma_{\mu,c}} \kappa^{p+1}\,ds\,,$$
	where $p=1/(1-\mu)$. Since the curvature $\kappa=\kappa_{\mu,c}$ of the generating curve $\gamma_{\mu,c}$ is symmetric and monotone on each branch of $\gamma_{\mu,c}$, varying from $0$ to its maximum value $\widehat{\alpha}>0$, we may use \eqref{EL} to make a change of variable in the above integral (see Appendix A for details). This then reads
	$$\int_{\gamma_{\mu,c}} \kappa^{p+1}\,ds=-\frac{2\mu}{1-\mu}\int_0^{\widehat{\alpha}} \frac{\kappa^{1/(1-\mu)}}{\sqrt{\widehat{Q}_{\mu,c}(\kappa)}}\,d\kappa\,,$$
	where $\widehat{\alpha}>0$ is the only positive solution of $\widehat{Q}_{\mu,c}(\kappa)=-\mu^2\kappa^2+c^2\kappa^{-2\mu/(1-\mu)}+1=0$. 
	The change of variable $\kappa=x^{1-\mu}$ then yields formula \eqref{ARhyperelliptic}, proving the first assertion.
	
	For the second assertion, observe that for every fixed $c_o>0$,
	$$\lim_{\mu\to-\infty} \mathcal{A}_R(M_{\mu,c_o})=-\frac{2\pi^2}{c_o}\,.$$
	Indeed, making the change of variable $x=\alpha y^{1/(1-\mu)}$ in \eqref{ARhyperelliptic} and setting $\beta=-\mu\alpha^{1-\mu}$, we obtain
	$$\mathcal{A}_R(M_{\mu,c_o})=-\frac{4\pi}{c_o}\cdot\frac{(-\mu)^{-\mu/(1-\mu)}}{1-\mu}\cdot\beta^{(2-\mu)/(1-\mu)}\int_0^1 \frac{y^{1/(1-\mu)}}{\sqrt{1-\beta^2y^2+c_o^2\left(\frac{\beta y}{-\mu}\right)^{-2\mu/(1-\mu)}}}\,dy\,.$$
	Since $Q_{\mu,c_o}(\alpha)=0$, the number $\beta$ satisfies
	$$1-\beta^2+c_o^2\left(\frac{\beta}{-\mu}\right)^{-2\mu/(1-\mu)}=0\,.$$
	It then follows that $\beta\to 1$ as $\mu\to-\infty$. Moreover,
	$$\frac{(-\mu)^{-\mu/(1-\mu)}}{1-\mu}\cdot\beta^{(2-\mu)/(1-\mu)}\longrightarrow 1\,,$$
	as $\mu\to-\infty$. The integrand above converges pointwise on $(0,1)$ to $1/\sqrt{1-y^2}$ and, for all sufficiently negative $\mu$, these integrands are bounded by an integrable function of the form $C/\sqrt{1-y^2}$. The dominated convergence theorem therefore gives
	$$\lim_{\mu\to-\infty}\mathcal{A}_R(M_{\mu,c_o})=-\frac{4\pi}{c_o}\int_0^1\frac{dy}{\sqrt{1-y^2}}=-\frac{2\pi^2}{c_o}\,.$$
	
	Finally, formula \eqref{ARhyperelliptic} shows that $\mathcal{A}_R(M_{\mu,c})$ depends continuously on $(\mu,c)\in(-\infty,0)\times(0,\infty)$, which is a connected domain. Its range is therefore an interval. Given any fixed $\rho<0$, choose $c_o=-2\pi^2/\rho>0$. By the above limit, the value $\rho$ belongs to the closure of the range. Hence, the closure of the range contains $(-\infty,0)$, while formula \eqref{ARhyperelliptic} directly shows that $\mathcal{A}_R(M_{\mu,c})<0$ for every $\mu<0$ and $c>0$, and so the range itself is contained in $(-\infty,0)$. Since the range is an interval, it must be exactly $(-\infty, 0)$. This completes the proof.
\end{proof}

\begin{rem} For every fixed $\mu<0$ and $c>0$, the improper integral in \eqref{ARhyperelliptic} converges. Indeed, $Q_{\mu,c}(0)=1$, so the integrand is integrable near $x=0$. Moreover, $Q_{\mu,c}'(\alpha)<0$, and hence there exist $\epsilon>0$ and $C>0$ such that  $$\frac{x^{1-\mu}}{\sqrt{Q_{\mu,c}(x)}} \leq \frac{C}{\sqrt{\alpha-x}}\,,$$ for every $x\in(\alpha-\varepsilon,\alpha)$. Therefore, $x^{1-\mu}/\sqrt{Q_{\mu,c}(x)}\in\mathcal{L}^1(0,\alpha)$. 
	
Formula \eqref{ARhyperelliptic} also yields the following asymptotic relations. For every fixed $\mu_o<0$, one has 
$$\lim_{c\to 0^+}\mathcal{A}_R(M_{\mu_o,c})=-\infty\,,\quad\quad\quad\text{and}\quad\quad\quad\lim_{c\to\infty}\mathcal{A}_R(M_{\mu_o,c})=-4\pi\,.$$ 
Likewise, for every fixed $c_o>0$,
$$\lim_{\mu\to 0^-}\mathcal{A}_R(M_{\mu,c_o})=-4\pi\frac{\sqrt{c_o^2+1}}{c_o}\,,\quad\quad\quad\text{and}\quad\quad\quad\lim_{\mu\to-\infty}\mathcal{A}_R(M_{\mu,c_o})=-\frac{2\pi^2}{c_o}\,.$$ 
These limits follow from suitable changes of variable in formula \eqref{ARhyperelliptic}.
\end{rem}

For general $\mu<0$, the integral in \eqref{ARhyperelliptic} is a generalized hyperelliptic integral, since $Q_{\mu,c}$ need not be a polynomial. For certain values of $\mu$, however, it reduces to an elliptic integral. A distinguished case is $\mu=-1$, for which \eqref{LW} reads $\kappa_2=-\kappa_1$ and so $M_{-1,c}$ is a spherical catenoid in $\mathbb{H}^3$. In this case, formula \eqref{ARhyperelliptic} agrees with the expression first obtained in Theorem 4.1 of \cite{PT} (see (20) of \cite{PT} as well) in terms of the standard complete elliptic integrals of the first (K) and second (E) kind,
$$\mathcal{A}_R(M_{-1,c})=\frac{4\pi}{\sqrt{2\zeta^2-1}}\left((1-\zeta^2)K(\zeta)-E(\zeta)\right),\quad\quad\quad \zeta=\sqrt{\frac{c^2+\sqrt{c^4+4}}{2\sqrt{c^4+4}}}\,.$$

The proportionality relation in Corollary \ref{t3} immediately converts formula \eqref{ARhyperelliptic} into an explicit expression for the Willmore energy $\mathcal{W}$. Determining its range requires, in addition, the sharp lower bound for the Willmore energy of the doubled rotational tori.

\begin{theorem}\label{cor}
	Let $\mu<0$ and $c>0$. The Willmore energy $\mathcal{W}$ of the rotational linear Weingarten annulus $M_{\mu,c}$, regarded as a surface in $\mathbb{R}^3_+$, is given by the convergent improper integral
	\begin{equation}\label{Whyperelliptic}
		\mathcal{W}(M_{\mu,c})=-\frac{\mu(\mu-1)^2\pi}{c}\int_0^\alpha \frac{x^{1-\mu}}{\sqrt{Q_{\mu,c}(x)}}\,dx\,,
	\end{equation}
	where $\alpha>0$ is the only positive solution of $Q_{\mu,c}(x)=-\mu^2x^{2(1-\mu)}+c^2x^{-2\mu}+1=0$.
	
	Moreover, as $\mu<0$ and $c>0$ vary, the Willmore energy $\mathcal{W}(M_{\mu,c})$ takes every value in $(\pi^2,\infty)$.
\end{theorem}
\begin{proof} 
	The expression \eqref{Whyperelliptic} follows immediately from Corollary \ref{t3} and Theorem \ref{t5}. Indeed,
	$$\mathcal{W}(M_{\mu,c})=\frac{(\mu-1)^2}{4\mu}\,\mathcal{A}_R(M_{\mu,c})=-\frac{\mu(\mu-1)^2\pi}{c}\int_0^\alpha \frac{x^{1-\mu}}{\sqrt{Q_{\mu,c}(x)}}\,dx\,,$$
	where $\alpha>0$ is the only positive solution of $Q_{\mu,c}(x)=-\mu^2x^{2(1-\mu)}+c^2x^{-2\mu}+1=0$.
	
	We next compute the range of the map $\mathcal{W}$ acting on the family $\{M_{\mu,c}\,\lvert\,\mu<0,c>0\}$. A change of variable analogous to that of the proof of Theorem \ref{t5} in combination with the dominated convergence theorem can be used to show that
	\begin{equation}\label{limit}
	\lim_{\mu\to-\infty} \mathcal{W}(M_{\mu,t\mu})=-\frac{\pi^2}{2t\sqrt{1-t^2}}\,,
	\end{equation}
	for any $t\in(-1,0)$. The function on the right-hand side of \eqref{limit} takes every value in $[\pi^2,\infty)$; its minimum $\pi^2$ is attained at $t=-1/\sqrt{2}$, while it tends to infinity as $t\to 0^-$ or $t\to-1^+$.
	
	It remains to prove that $\pi^2$ is a strict lower bound for $\mathcal{W}(M_{\mu,c})$. Since $M_{\mu,c}$ meets the plane $\{z=0\}$ orthogonally, reflection across this plane produces a $\mathcal{C}^2$ rotational torus $\widetilde{M}_{\mu,c}$ in $\mathbb{R}^3$. By symmetry, $\mathcal{W}(\widetilde{M}_{\mu,c})=2\mathcal{W}(M_{\mu,c})$. The sharp Willmore inequality for axially symmetric tori \cite{LS} (see also \cite{MN} for the general case, without any symmetry assumption) therefore gives $\mathcal{W}(M_{\mu,c})\geq\pi^2$. Equality would imply that $\widetilde{M}_{\mu,c}$ is, up to a Euclidean similarity, the rotational Clifford torus. After employing its standard parameterization 
	$$X(s,t)=\left((\sqrt{2}+\cos s)\cos t, (\sqrt{2}+\cos s)\sin t, \sin s\right),$$
	we compute the (Euclidean) principal curvatures of the rotational Clifford torus with respect to the outward pointing normal
	$$\overline{\xi}=\left(\cos s \cos t, \cos s \sin t, \sin s\right),$$
	obtaining
	$$\overline{\kappa}_1=-1\,,\quad\quad\quad \overline{\kappa}_2=-\frac{\cos s}{\sqrt{2}+\cos s}\,.$$
	Conformally transforming its upper half to the hyperbolic metric, we get from $\kappa_i=z\overline{\kappa}_i+\overline{\xi}_z$ the (hyperbolic) principal curvatures
	$$\kappa_1=0\,,\quad\quad\quad \kappa_2=\frac{\sqrt{2}\sin s}{\sqrt{2}+\cos s}\,,$$
	which do not satisfy the relation \eqref{LW}. Thus, equality in the Willmore inequality cannot occur for any $M_{\mu,c}$. Consequently, $\mathcal{W}(M_{\mu,c})>\pi^2$, for every $\mu<0$ and $c>0$.
	
	Finally, the same argument as in the last part of the proof of Theorem \ref{t5} shows that the range of $\mathcal{W}(M_{\mu,c})$ is exactly $(\pi^2,\infty)$.
\end{proof}

Doubling the annuli $M_{\mu,c}$ across the plane $\{z=0\}$ produces a family of $\mathcal{C}^2$ rotational tori $\widetilde{M}_{\mu,c}$ whose Willmore energies realize every value in $(2\pi^2,\infty)$. The optimal value $2\pi^2$ is not attained within this family, but it is approached along the subfamily $c=-\mu/\sqrt{2}$ as $\mu\to-\infty$. Since the Willmore functional is quadratic in the second fundamental form, $\mathcal{W}^{2,2}$ is the natural regularity in which to expect the tori to converge as their Willmore energies converge. The following result establishes the stronger conclusion that these tori converge in $\mathcal{W}^{2,p}$ for every $1\leq p<\infty$.

\begin{theorem}\label{propnew} Let $\mu<0$ and $c=-\mu/\sqrt{2}$. As $\mu\to-\infty$, the family of rotational tori $\widetilde{M}_{\mu,c}$ converges in $\mathcal{W}^{2,p}$, for every $1\leq p<\infty$, to the rotational Clifford torus. The convergence is smooth locally away from the reflection plane $\{z=0\}$, but cannot be strengthened to global convergence in $\mathcal{W}^{2,\infty}$.
\end{theorem}
\begin{proof}
	The rotational tori $\widetilde{M}_{\mu,c}$ are obtained by reflecting the annuli $M_{\mu,c}$ across the plane $\{z=0\}$. It therefore suffices to prove the convergence of the generating curves $\gamma_{\mu,c}$ of $M_{\mu,c}$.
	
	By Lemma \ref{lemma1}, the generating curves $\gamma_{\mu,c}$ converge smoothly locally away from their endpoints to the upper semicircle of radius one centered at $(\sqrt{2},0,0)$. By Lemma \ref{lemma2}, this convergence extends to global convergence in $\mathcal{W}^{2,p}$, for every $1\leq p<\infty$, of the doubled generating curves. Rotating the limiting circle about the $z$-axis gives the rotational Clifford torus. The curvature obstruction in Lemma \ref{lemma3} rules out global convergence in $\mathcal{W}^{2,\infty}$.
\end{proof}

\begin{rem} 
By the Sobolev embedding theorem (see, for instance, \cite{M}), in our setting, convergence in $\mathcal{W}^{2,p}$ for every $1\leq p<\infty$ implies convergence in $\mathcal{C}^{1,\alpha}$ for every $\alpha\in(0,1)$. In particular, the convergence is globally $\mathcal{C}^1$. The $\mathcal{C}^1$ global convergence can also be deduced directly from the proof (see Lemma \ref{lemma2}). On the other hand, the curvature mismatch in Lemma \ref{lemma3} rules out convergence in $\mathcal{C}^{1,1}$, and hence also in $\mathcal{C}^2$.
\end{rem}

The convergence established in Theorem \ref{propnew} is illustrated in Figure \ref{F0}, which shows three tori $\widetilde{M}_{\mu,c}$ with $c=-\mu/\sqrt{2}$. As $\mu\to-\infty$, their generating curves converge in $\mathcal{W}^{2,p}$ to the circle of radius one centered at $(\sqrt{2},0,0)$, and the corresponding tori converge in $\mathcal{W}^{2,p}$ to the rotational Clifford torus.

\begin{figure}[h!]
	\centering
	\includegraphics[height=5.25cm]{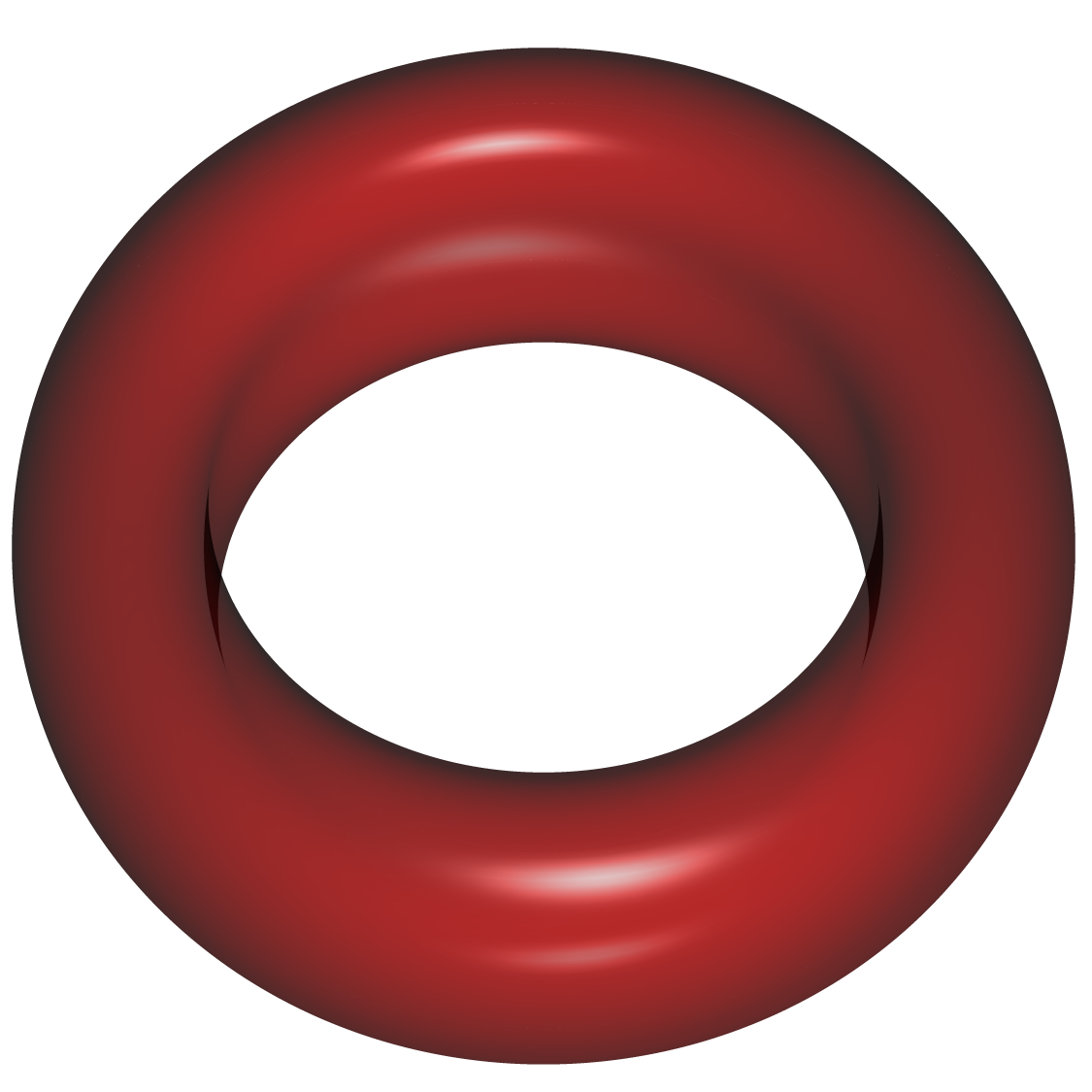}\,\,\,
	\includegraphics[height=5.25cm]{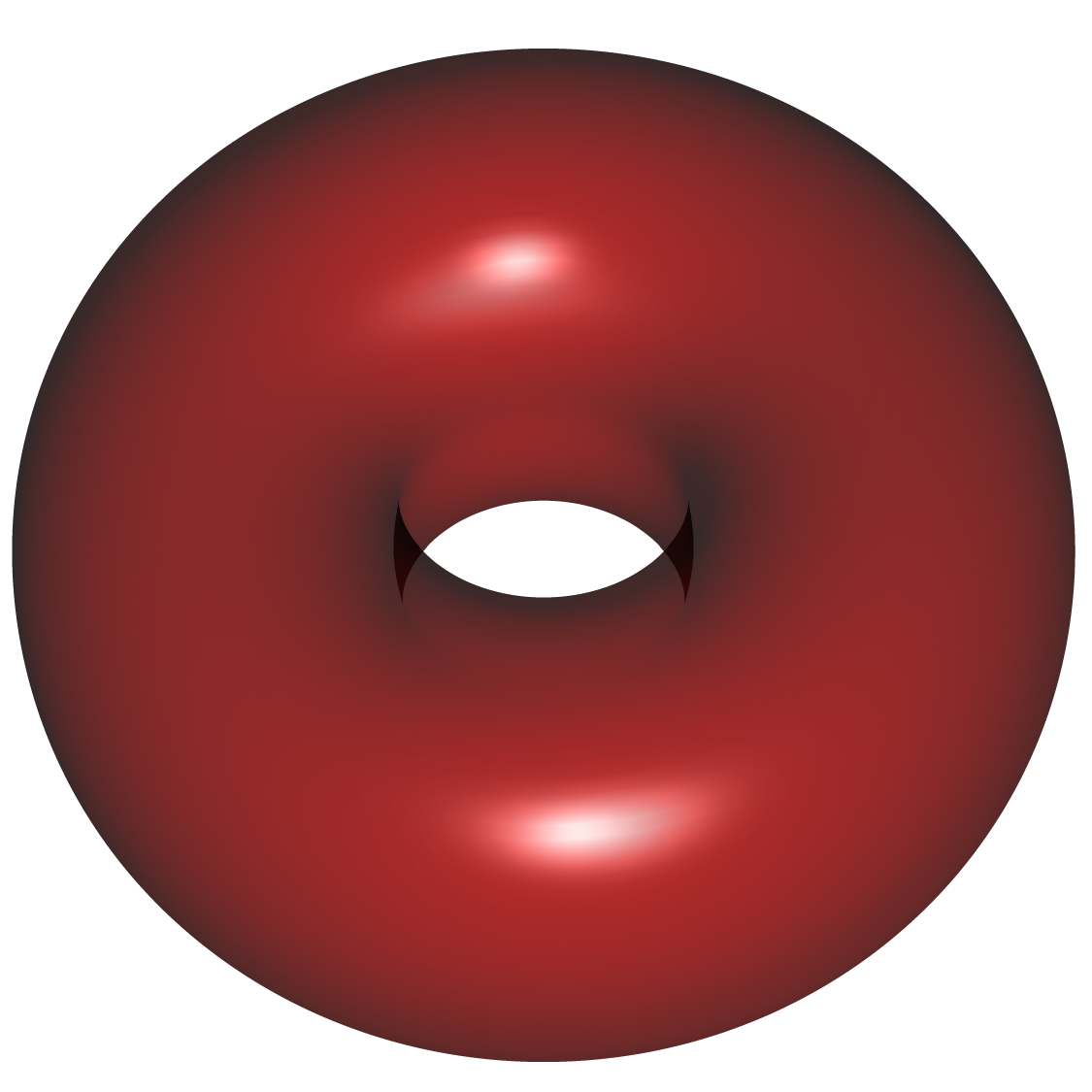}\,\,\,
	\includegraphics[height=5.25cm]{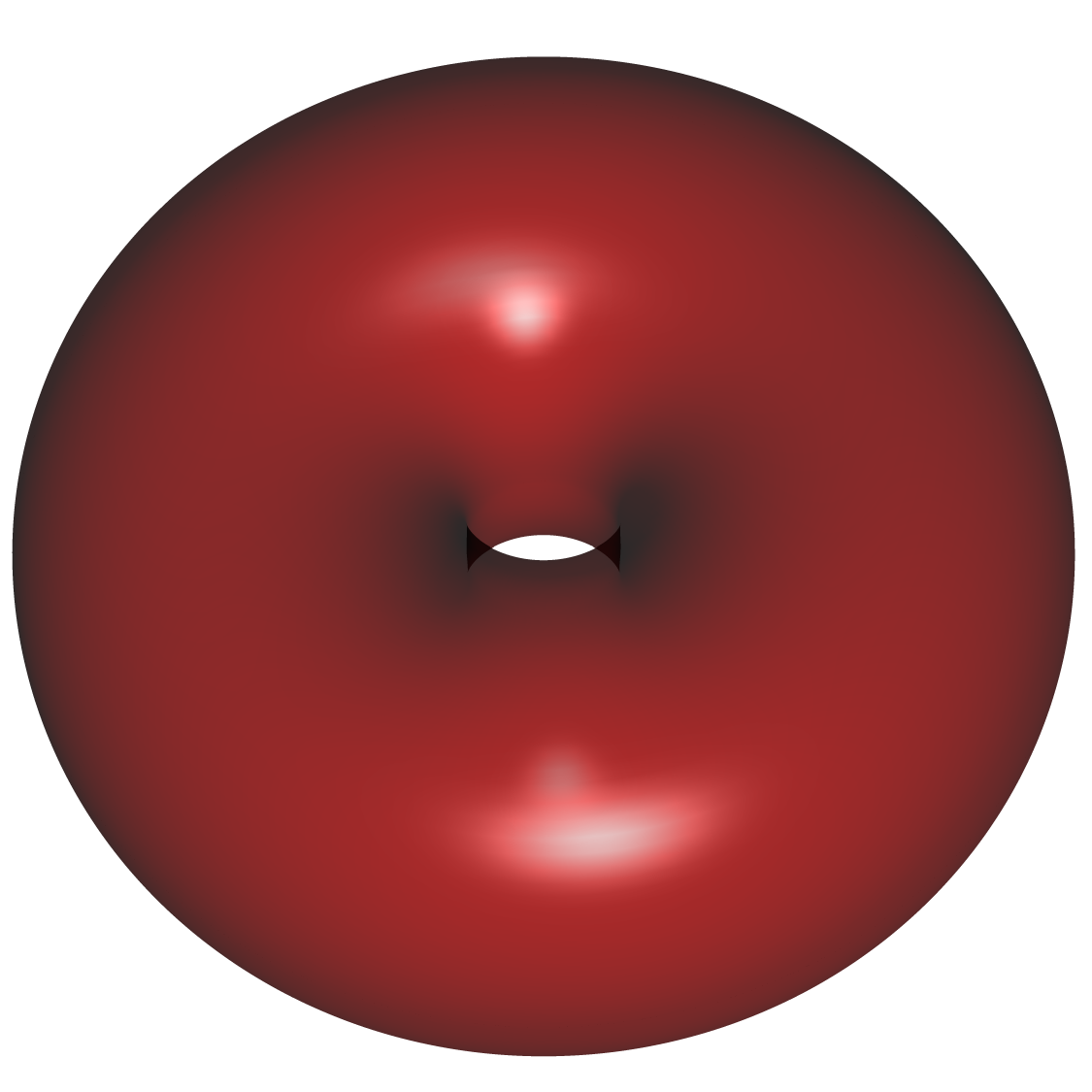}
	\caption{\small{Three rotational tori $\widetilde{M}_{\mu,c}$, with $c=-\mu/\sqrt{2}$, obtained by reflecting the annuli $M_{\mu,c}$ across the plane $\{z=0\}$. From left to right: $\mu=-0.6$, $\mu=-2$, and $\mu=-6$, respectively. The family converges in $\mathcal{W}^{2,p}$ to the rotational Clifford torus as $\mu\to-\infty$.}}
	\label{F0}
\end{figure}

In conclusion, the doubled linear Weingarten annuli $\widetilde{M}_{\mu,c}$ provide a geometrically distinguished family of rotational tori realizing every Willmore energy in $(2\pi^2,\infty)$. Although the optimal value $2\pi^2$ is not attained within the family, it is approached both in Willmore energy and geometrically through $\mathcal{W}^{2,p}$, for every $1\leq p<\infty$, convergence to the Clifford torus.

\section{Appendix A. Generalized Elastic Curves in $\mathbb{H}^2$}

The purpose of this appendix is to describe the geometric properties of the generating curves of the rotational linear Weingarten surfaces considered in Sections \ref{LWA} and \ref{5}.

We use the hyperboloid model of the hyperbolic plane. Let $(u,v,w)$ be the standard coordinates on $\mathbb{R}^3$, endowed with the Lorentzian metric $\widetilde{g}=du^2+dv^2-dw^2$. The resulting Lorentz--Minkowski space is denoted by $\mathbb{L}^3$. The hyperbolic plane is realized as the upper sheet of the two-sheeted hyperboloid 
$$\mathbb{H}^2 = \{ (u,v,w)\in\mathbb{L}^3 \,\lvert\, u^2+v^2-w^2=-1, w>0\}\,,$$
endowed with the Riemannian metric induced by $\widetilde{g}$.

For visualization purposes, this model can be identified with the Poincar\'e disk by means of the isometry 
$$(u,v,w) \longmapsto \frac{1}{1+w}(u,v)\,.$$ 
It can also be identified with the upper half-plane model $\mathbb{H}^2=\{ (x,0,z)\in\mathbb{H}^3 \,\lvert\, z>0\}$ used throughout the paper by means of
$$(u,v,w) \longmapsto \frac{1}{w-v}(u,0,1)\,.$$

Let $\gamma:I\longrightarrow\mathbb{H}^2$ be an arc length parameterized curve with positive, non-constant curvature $\kappa$. For constants $\mu<0$ and $c>0$, suppose that $\kappa$ satisfies
\begin{equation}\label{ELrepeat}
	\left(\kappa_s\right)^2=\frac{(1-\mu)^2\kappa^2}{\mu^2}\,\widehat{Q}_{\mu,c}(\kappa)\,,
\end{equation}
where $\widehat{Q}_{\mu,c}(\kappa)=-\mu^2\kappa^2+c^2\kappa^{-2\mu/(1-\mu)}+1$. Such a curve will be called a generalized elastic curve.

An elementary analysis of $\widehat{Q}_{\mu,c}$ shows that it has a unique positive zero $\widehat{\alpha}>0$ and that $\kappa^2\widehat{Q}_{\mu,c}(\kappa)\geq 0$ if and only if $\kappa\in[0,\widehat{\alpha}]$. Thus, the curvature of a generalized elastic curve takes values in $[0,\widehat{\alpha}]$. After choosing the origin of the arc length parameter so that $\kappa(0)=\widehat{\alpha}$, uniqueness for the non-constant solution of the corresponding initial value problem shows that $\kappa$ is even. Moreover, $\kappa$ increases strictly from $0$ to $\widehat{\alpha}$ on $(-\infty,0)$ and decreases strictly from $\widehat{\alpha}$ to $0$ on $(0,\infty)$. In particular, 
$$\lim_{s\to\pm\infty}\kappa(s)=0\,.$$
The symmetry and monotonicity of $\kappa$ justify the change of variable used in the proof of Theorem \ref{t5}.

Adapting the computations of \cite{LS2} (see also Section 3.2.1 of \cite{Pa}) to the present case, one concludes that a generalized elastic curve in $\mathbb{H}^2$ can be parameterized in terms of its arc length parameter, up to an ambient isometry, as
$$\gamma_{\mu,c}(s)=\frac{1}{c}\left(\kappa^{\mu/(1-\mu)},\sqrt{c^2+\kappa^{2\mu/(1-\mu)}}\,\sinh \phi(s),\sqrt{c^2+\kappa^{2\mu/(1-\mu)}}\,\cosh \phi(s)\right),$$
where $\kappa=\kappa_{\mu,c}(s)$ is the curvature of $\gamma_{\mu,c}$ and
$$\phi(s)=-c\mu \int_0^s \frac{\kappa^{1/(1-\mu)}}{c^2+\kappa^{2\mu/(1-\mu)}}\,d\sigma\,.$$
The function $\phi$ describes the hyperbolic angular progression of the curve.

Since $\kappa_{\mu,c}$ is even and the integrand defining $\phi$ is even, the function $\phi$ is odd. It follows from the above parameterization that $\gamma_{\mu,c}(-s)=(u(s),-v(s),w(s))$. Thus, $\gamma_{\mu,c}$ is symmetric with respect to the geodesic $\{v=0\}\cap\mathbb{H}^2$, and its two branches correspond to $s<0$ and $s>0$, respectively. Moreover, 
$$\phi_s=-c\mu \frac{\kappa^{1/(1-\mu)}} {c^2+\kappa^{2\mu/(1-\mu)}} >0\,,$$ 
since $\mu<0$ and $c>0$. Hence $\phi$ is strictly increasing, and the sign of the second coordinate of $\gamma_{\mu,c}(s)$ agrees with the sign of $s$. In addition, the first coordinate 
$$u(s) = \frac{1}{c} \kappa(s)^{\mu/(1-\mu)}\,,$$ 
is strictly monotone on each branch. It follows that each branch is embedded and that the two branches meet only at $s=0$. Therefore, $\gamma_{\mu,c}$ is embedded.

Furthermore, since the first coordinate of $\gamma_{\mu,c}$ is strictly positive for every $s\in\mathbb{R}$, the curve does not intersect the geodesic $\{u=0\}\cap\mathbb H^2$, which corresponds to the axis of rotation. Recall that $\kappa_{\mu,c}(s)\to 0$ as $s\to\pm\infty$ and that $\mu/(1-\mu)<0$, hence, the parameterization also shows that the two branches escape every compact subset of $\mathbb{H}^2$ and converge to two symmetric points of its ideal boundary. Thus, $\gamma_{\mu,c}$ is complete and properly embedded.

Projecting the curve to the Poincar\'e disk and taking the limit as $s\to\pm\infty$, one finds that its tangent vector becomes parallel to the radial direction at each ideal endpoint. Hence, $\gamma_{\mu,c}$ approaches the ideal boundary at a right angle at both ends. Under the standard identification with the upper half-plane model, the curve therefore approaches $\{z=0\}$ at a right angle.

\begin{figure}[h!]
	\centering
	\includegraphics[height=5.58cm]{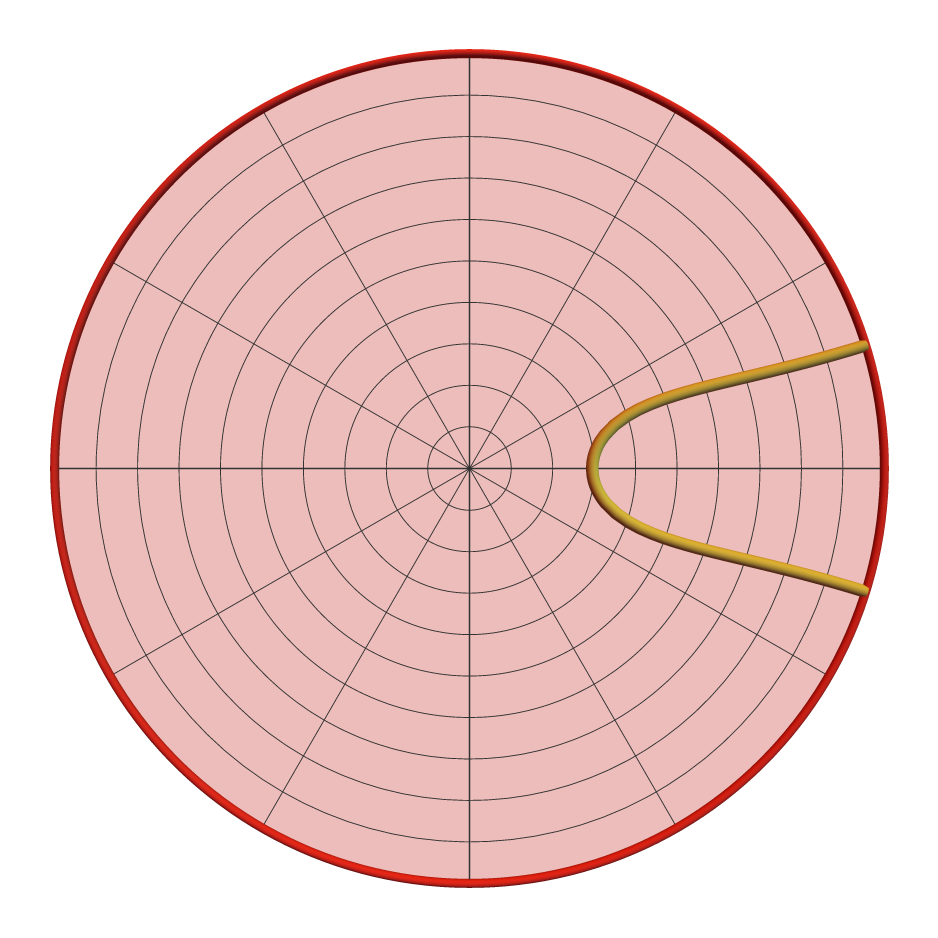}
	\includegraphics[height=5.58cm]{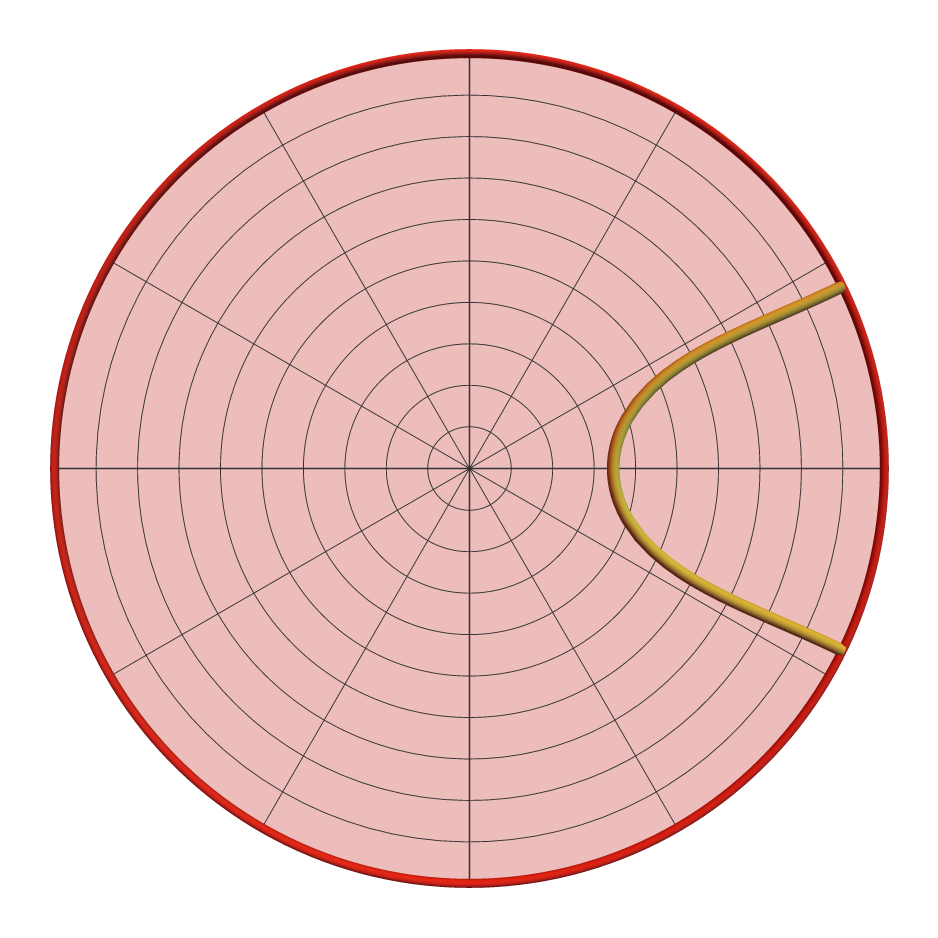}
	\includegraphics[height=5.58cm]{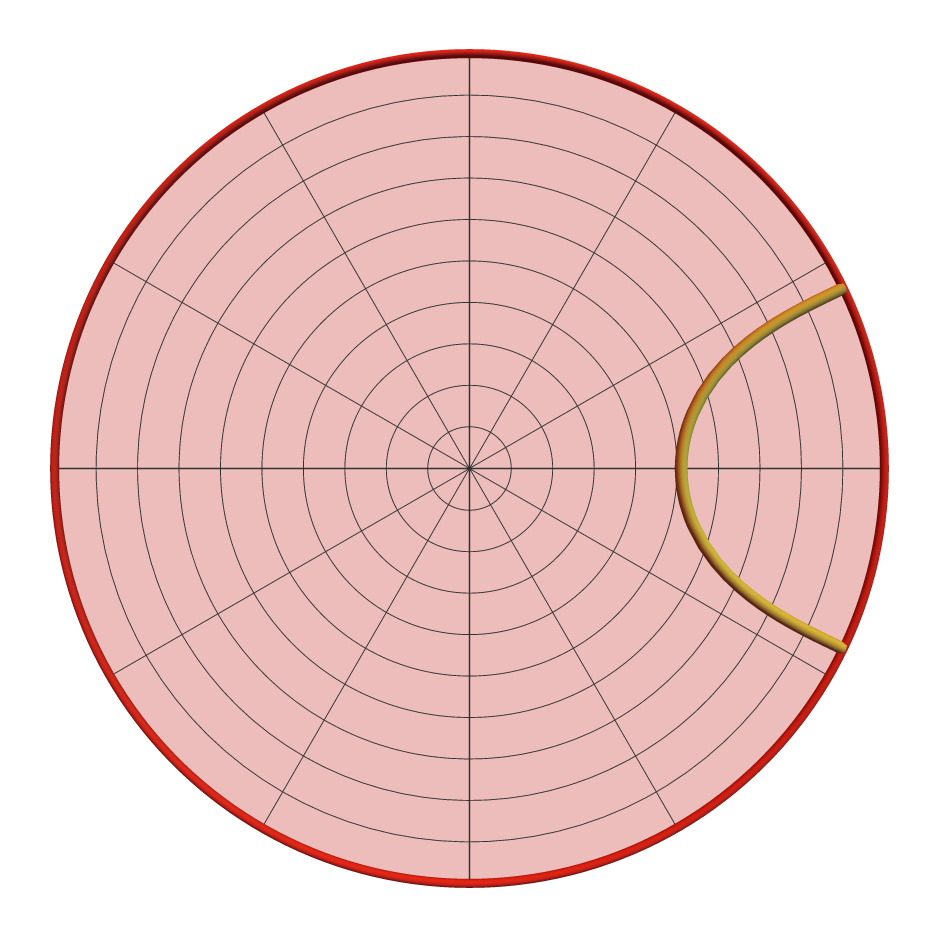}
	\caption{\small{Generalized elastic curves $\gamma_{\mu,c}$ represented in the Poincar\'e disk model of $\mathbb{H}^2$. From left to right, the parameters are $(\mu,c)=(-1/2,1)$, $(\mu,c)=(-1,1)$, and $(\mu,c)=(-2,1)$, respectively.}}
	\label{F1}
\end{figure}

Figure \ref{F1} shows three generalized elastic curves $\gamma_{\mu,c}$ in the Poincar\'e disk model of $\mathbb{H}^2$. It illustrates their symmetry and embeddedness, the decay of their curvatures toward both ends, and their orthogonal approach to the ideal boundary $\partial_\infty\mathbb{H}^2$.

We now rotate $\gamma_{\mu,c}$ about the geodesic $\{u=0\}\cap\mathbb{H}^2$. Identifying $\gamma_{\mu,c}(s) = (u(s),v(s),w(s)) \in\mathbb{H}^2$ with $\gamma_{\mu,c}(s) = (u(s),0,v(s),w(s)) \in\mathbb{H}^3\subset\mathbb{L}^4$, the associated rotational surface $M_{\mu,c}$ is parameterized by
$$X_{\mu,c}(s,\theta)=\mathcal{R}(\theta)\cdot\gamma_{\mu,c}(s)\subset\mathbb{H}^3\subset\mathbb{L}^4\,,$$
where $\mathcal{R}(\theta)$ is the $4\times 4$ orthogonal matrix
$$\mathcal{R}(\theta)=\begin{pmatrix} \cos\theta & -\sin\theta & 0 & 0 \\ \sin\theta & \cos\theta & 0 & 0 \\ 0 & 0 & 1 & 0 \\ 0 & 0 & 0 & 1\end{pmatrix}.$$
Since $\gamma_{\mu,c}$ is complete, properly embedded, disjoint from the axis of rotation, and approaches the ideal boundary at a right angle at both ends, its rotational surface $M_{\mu,c}$ is a properly embedded topological annulus approaching $\partial_\infty\mathbb{H}^3$ at a right angle. This proves the first assertion of Theorem \ref{prop}.

The rotational annuli $M_{\mu,c}$ generated by the curves in Figure \ref{F1} are shown in Figure \ref{F2} after identifying the hyperboloid model of $\mathbb{H}^3$ with the Poincar\'e ball model. These representations illustrate the proper embeddedness and orthogonal boundary behavior established above. The first surface $M_{\mu,c}$, corresponding to $(\mu,c)=(-1/2,1)$, is also displayed in the upper half-space model of $\mathbb{H}^3$ on the front page.

\begin{figure}[h!]
	\centering
	\includegraphics[height=5.5cm]{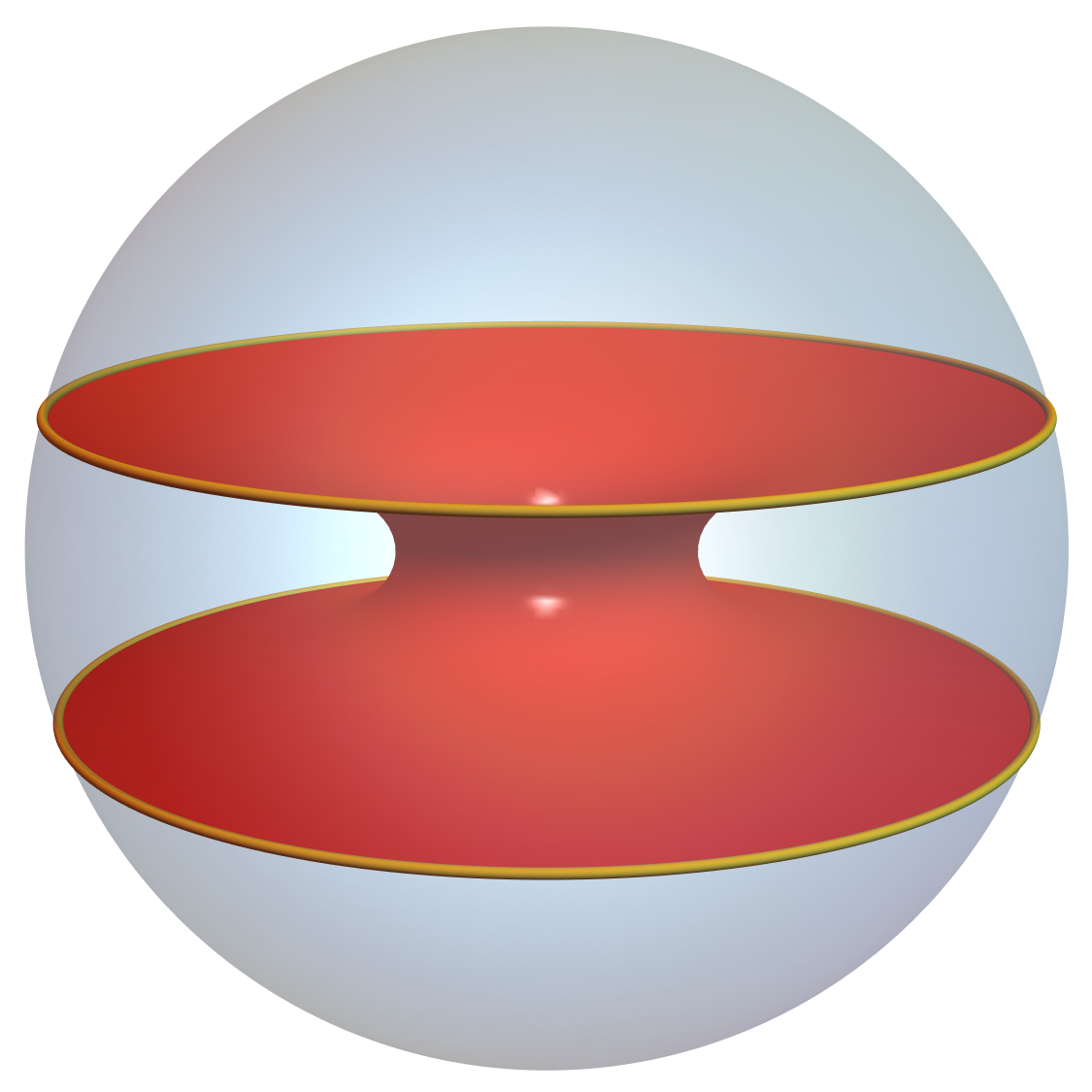}\,
	\includegraphics[height=5.5cm]{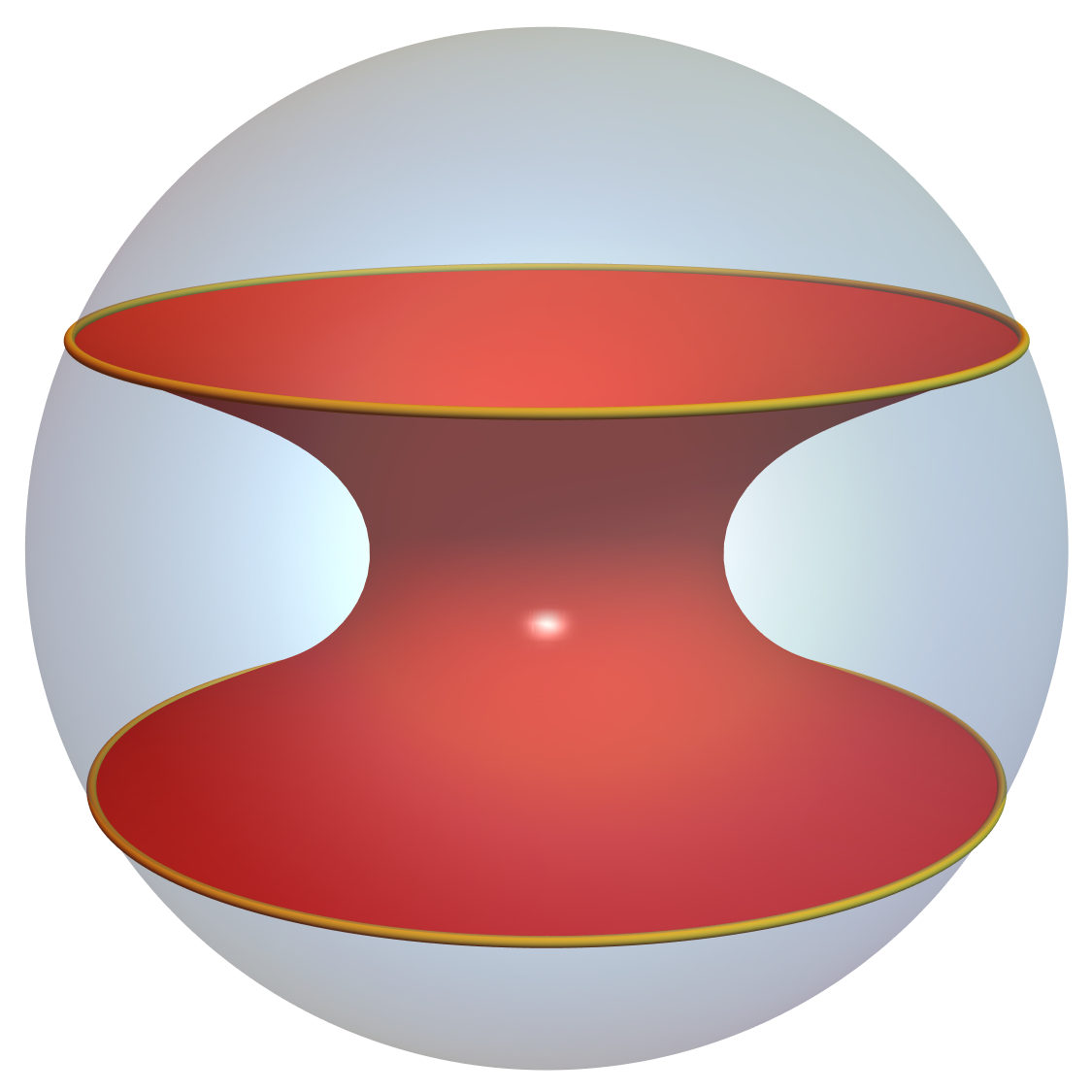}\,
	\includegraphics[height=5.5cm]{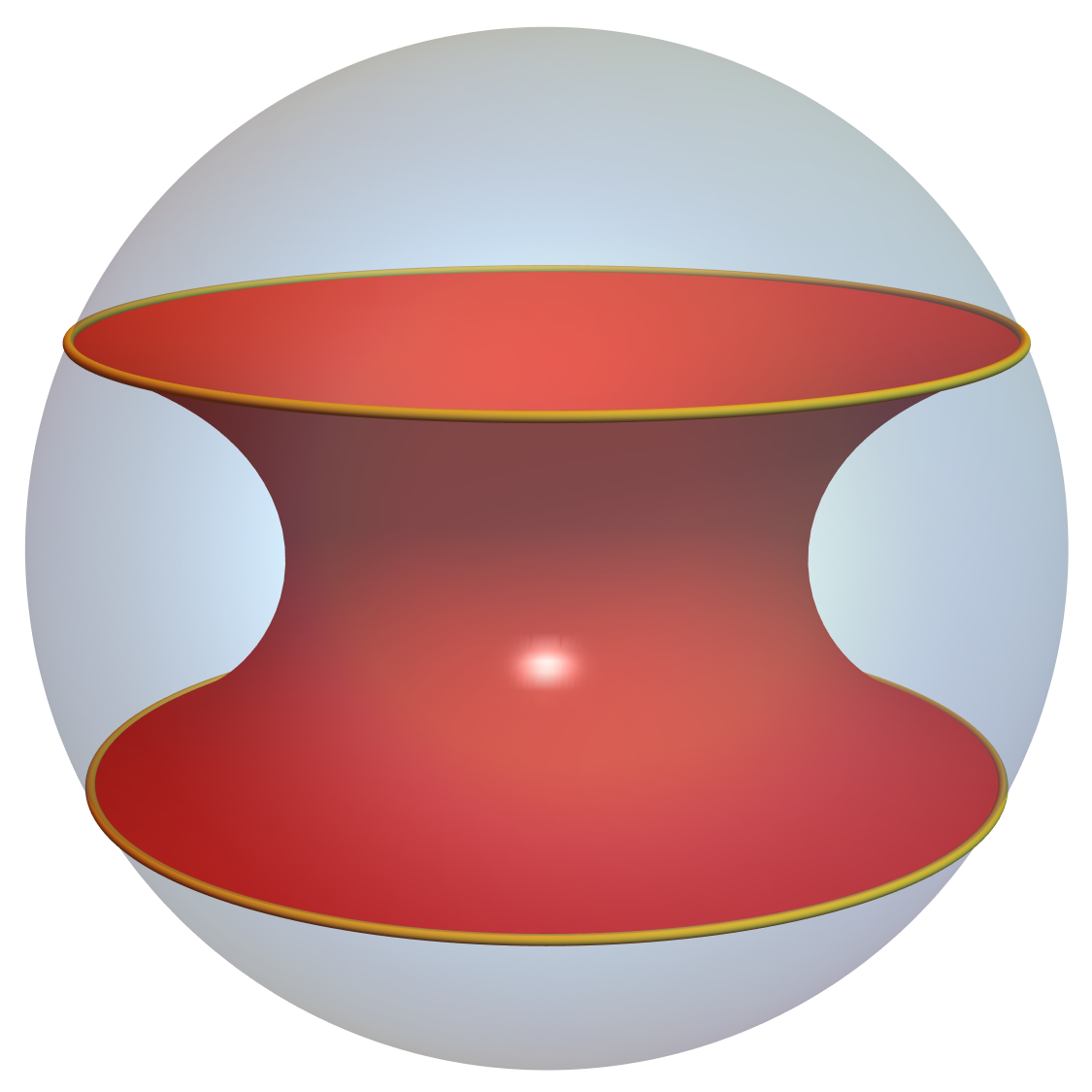}
	\caption{\small{The rotational annuli $M_{\mu,c}$ in the Poincar\'e ball model of $\mathbb{H}^3$, generated by the generalized elastic curves shown in Figure \ref{F1}. From left to right, the parameters are $(\mu,c)=(-1/2,1)$, $(\mu,c)=(-1,1)$, and $(\mu,c)=(-2,1)$.}}
	\label{F2}
\end{figure}

\section{Appendix B. Convergence of a Class of Generalized Elastic Curves}

The purpose of this appendix is to establish the convergence properties of the generalized elastic curves associated with the subfamily $c=-\mu/\sqrt{2}$ as $\mu\to-\infty$. These results provide the technical ingredients needed to prove Theorem \ref{propnew}: smooth convergence away from the endpoints, global convergence in $\mathcal{W}^{2,p}$ for every $1\leq p<\infty$, and the failure of convergence in $\mathcal{W}^{2,\infty}$ at the endpoints. Throughout this appendix, we assume that $\mu<0$ and set $c=-\mu/\sqrt{2}$. 

We begin by determining the limiting curve away from its endpoints. Using the arc length parameter $s$ introduced in Appendix A, we show that the generalized elastic curves $\gamma_{\mu,c}(s)$ converge smoothly on every compact interval to a semicircle.

\begin{lemma}\label{lemma1}
	Let $\mu<0$ and $c=-\mu/\sqrt{2}$. As $\mu\to-\infty$, the generalized elastic curves $\gamma_{\mu,c}$ converge in $\mathcal{C}^\infty$ on every compact interval to
	$$\widetilde{\gamma}(s)=\frac{1}{\sqrt{2}-\tanh s}\left(1,0,{\rm sech}\,s\right),$$
	which parameterizes the upper semicircle of radius one centered at $(\sqrt{2},0,0)$.
\end{lemma}
\begin{proof}
	After identifying the hyperboloid model of $\mathbb{H}^2$ with the upper half-plane model, it follows from the discussion in Appendix A that the curves $\gamma_{\mu,c}$ can be parameterized as
	$$\gamma_{\mu,c}(s)=\frac{1}{\sqrt{1+f^2}\left(\cosh\phi-\sinh\phi\right)}\left(f,0,1\right)=\frac{e^\phi}{\sqrt{1+f^2}}\left(f,0,1\right),$$
	where 
	$$\phi=-\int_0^s \frac{\mu \kappa f}{1+f^2}\,d\sigma\,,$$
	and, by \eqref{f} and the choice $c=-\mu/\sqrt{2}$,
	\begin{equation}\label{fproof}
		f=\frac{\sqrt{2}}{-\mu}\,\kappa^{\mu/(1-\mu)}\,.
	\end{equation}
	Throughout the proof, we suppress the dependence on $\mu$ for simplicity.
	
	The curvature $\kappa_{\mu,c}=\kappa$ of the generating curves $\gamma_{\mu,c}$ is the solution of the first-order ordinary differential equation \eqref{EL} with $c=-\mu/\sqrt{2}$ and initial condition $\kappa(0)=\widehat{\alpha}$, where $\widehat{\alpha}$ is its maximum value. Setting $x=\mu\kappa$, this differential equation becomes
	$$(x_s)^2=G_\mu(x)\,$$ 
	where 
	$$G_\mu(x) = \frac{(1-\mu)^2x^2}{\mu^2} \left( 1-x^2 +\frac{x^2}{2} \left(\frac{x}{\mu}\right)^{-2/(1-\mu)} \right).$$ 
	Along the non-constant solution, differentiating the identity $(x_s)^2=G_\mu(x)$ gives 
	$$x_{ss}=\frac{1}{2}G_\mu'(x)\,.$$ 
	This equation extends continuously to $s=0$. As $\mu\to-\infty$, the functions $G_\mu$ converge smoothly on compact subsets of $(-\infty,0)$ to 
	$$G(x)=x^2\left(1-\frac{x^2}{2}\right).$$ 
	Moreover, $x(0)=\mu\,\widehat{\alpha}\longrightarrow-\sqrt{2}$, and $x_s(0)=0$. Therefore, by smooth dependence of solutions on the parameter and initial data, $x$ converges in $\mathcal{C}^\infty$ on compact intervals to the unique solution of the initial value problem
	$$x_{ss}=\frac{1}{2}G'(x)=x-x^3\,, \quad\quad\quad x(0)=-\sqrt{2}\,, \quad x_s(0)=0\,.$$ 
	This solution is $x(s)=-\sqrt{2}\,{\rm sech}\,s$. Consequently, $\mu\kappa(s)\to -\sqrt{2}\,{\rm sech}\,s$.
	
	We next observe that 
	$$\kappa^{1/(1-\mu)}\longrightarrow 1\,.$$ 
	Indeed, 
	$$\log\left(\kappa^{1/(1-\mu)}\right)=\frac{1}{1-\mu}\,\log(\kappa)=\frac{1}{1-\mu}\,\log\left(\frac{x}{\mu}\right) = \frac{1}{1-\mu}\left(\log(-x)-\log(-\mu)\right) \longrightarrow 0\,,$$ 
	since $x<0$ remains bounded away from zero on every fixed compact interval. It follows that 
	$$f(s) = -\frac{\sqrt{2}}{\mu\kappa}\, \kappa^{1/(1-\mu)} \longrightarrow \cosh s\,.$$ 
	Moreover, 
	$$\phi_s = -\frac{\mu\kappa f}{1+f^2} \longrightarrow \frac{\sqrt{2}}{1+\cosh^2s}\,.$$ 
	Since $\phi(0)=0$, we obtain 
	$$\phi(s) \longrightarrow \widetilde{\phi}(s) = \sqrt{2} \int_0^s \frac{d\sigma}{1+\cosh^2\sigma} = {\rm arctanh} \left( \frac{1}{\sqrt{2}}\,{\rm tanh}\,s \right).$$
	All the limits above hold in $\mathcal{C}^\infty$ on compact intervals. 
	
	Therefore, the curves $\gamma_{\mu,c}$ converge in $\mathcal{C}^\infty$ on compact intervals to
	$$\widetilde{\gamma}(s)=\frac{e^{\widetilde{\phi}(s)}}{\sqrt{1+\cosh^2s}}\left(\cosh s, 0, 1\right)=\frac{1}{\sqrt{2}-\tanh s}\left(1,0,{\rm sech}\,s\right),$$
	where the last equality follows after straightforward simplifications. A direct computation gives
	$$\left(\frac{1}{\sqrt{2}-\tanh s}-\sqrt{2}\right)^2+\left(\frac{{\rm sech}\,s}{\sqrt{2}-\tanh s}\right)^2=1\,,$$
	which means that $\widetilde{\gamma}$ is the upper semicircle of radius one centered at $(\sqrt{2},0,0)$.
\end{proof}

Lemma \ref{lemma1} establishes smooth convergence on every compact interval and hence locally away from the endpoints, which correspond to $s\to\pm\infty$. The arc length parameter $s$ is not well suited to studying the convergence at these endpoints because they occur at infinite parameter values. We therefore introduce a new parameter that places both endpoints at a finite value. The next lemma shows that the convergence extends to the endpoints in $\mathcal{W}^{2,p}$, for every $1\leq p<\infty$.

\begin{lemma}\label{lemma2}
	Let $\mu<0$ and $c=-\mu/\sqrt{2}$. As $\mu\to-\infty$, the generalized elastic curves $\gamma_{\mu,c}$ converge in $\mathcal{W}^{2,p}$, for every $1\leq p<\infty$, up to their endpoints, to the upper semicircle of radius one centered at $(\sqrt{2},0,0)$.
\end{lemma}
\begin{proof}
 By Lemma \ref{lemma1}, the curves already converge smoothly, and hence in $\mathcal{W}^{2,p}$, away from their endpoints. We now regard the curves as Euclidean curves in the closed upper half-plane. It remains to prove convergence near the endpoints, which lie in the plane $\{z=0\}$. We introduce the parameter
 $$\tau=\frac{1}{f(s)}\,,$$
 so that each endpoint corresponds to $\tau=0$, and
 \begin{equation}\label{gammatau}
 	\gamma_{\mu,c}(\tau)=\frac{e^{\phi(\tau)}}{\sqrt{1+\tau^2}}\left(1,0,\tau\right).
 \end{equation}
 We will prove convergence in $\mathcal{W}^{2,p}(0,\epsilon)$, for some sufficiently small $\epsilon>0$, by showing that $\phi$ and $\phi_\tau$ converge uniformly on $[0,\epsilon]$, while $\phi_{\tau\tau}$ converges in $\mathcal{L}^p(0,\epsilon)$. Indeed, differentiating the above parameterization shows that the curves and their first derivatives depend smoothly on $\phi$ and $\phi_\tau$, whereas their second derivatives depend smoothly on these quantities and linearly on $\phi_{\tau\tau}$. Consequently, the stated convergence properties imply uniform convergence of the curves and their first derivatives, together with convergence in $\mathcal{L}^p(0,\epsilon)$ of their second derivatives, and hence convergence in $\mathcal{W}^{2,p}(0,\epsilon)$.
 
 We first consider the branch $s>0$. From \eqref{fproof} and $x=\mu\kappa$, we get
 \begin{equation}\label{x}
 	x=\mu\left(\frac{\sqrt{2}}{-\mu}\right)^{1-1/\mu}\tau^{1-1/\mu}.
 \end{equation}
 For every sufficiently small fixed $\epsilon>0$, $x\to-\sqrt{2}\,\tau$ uniformly on $[0,\epsilon]$.
 
 Differentiating \eqref{x}, we get
 $$x_\tau=(\mu-1)\left(\frac{\sqrt{2}}{-\mu}\right)^{1-1/\mu}\tau^{-1/\mu}.$$
 For every $\tau>0$, this expression converges to $-\sqrt{2}$. Moreover, $x_\tau$ is uniformly bounded on $(0,\epsilon)$ for all sufficiently negative $\mu$. Therefore, by the dominated convergence theorem, $x_\tau\to-\sqrt{2}$ in $\mathcal{L}^p(0,\epsilon)$, for every $1\leq p<\infty$.
 
 We next express $\phi$ in terms of $\tau$. From the first integral \eqref{EL}, \eqref{fproof}, and $\tau=1/f$, we obtain
 $$\tau_s=-\tau\sqrt{1-x^2+\tau^2}\,,$$
 since $\tau$ is decreasing on the branch $s>0$. On the other hand,
 $$\phi_s=-\frac{x\tau}{1+\tau^2}\,.$$
 Combining both identities, we have
 $$\phi_\tau=\frac{x}{(1+\tau^2)\sqrt{1-x^2+\tau^2}}\longrightarrow \widetilde{\phi}_\tau=-\frac{\sqrt{2}\,\tau}{(1+\tau^2)\sqrt{1-\tau^2}}\,,$$
 uniformly on $[0,\epsilon]$. Here, $\epsilon>0$ is chosen sufficiently small so that $1-x^2+\tau^2$ is uniformly bounded away from zero on $[0,\epsilon]$.
 
 Differentiating the expression for $\phi_\tau$ and simplifying, we get
 $$\phi_{\tau\tau}=\frac{x_\tau}{(1-x^2+\tau^2)^{3/2}}-\frac{\tau x(3-2x^2+3\tau^2)}{(1+\tau^2)^2(1-x^2+\tau^2)^{3/2}}\,.$$
 Therefore, the uniform convergence of $x$, together with the convergence of $x_\tau$ in $\mathcal{L}^p(0,\epsilon)$, gives
 $$\phi_{\tau\tau}\longrightarrow \widetilde{\phi}_{\tau\tau}=\left(-\frac{\sqrt{2}\,\tau}{(1+\tau^2)\sqrt{1-\tau^2}}\right)_\tau,$$
 in $\mathcal{L}^p(0,\epsilon)$, for every $1\leq p<\infty$.
 
 At the fixed value $\tau=\epsilon$, the convergence $\phi(\epsilon)\to\widetilde{\phi}(\epsilon)$ follows from the smooth convergence away from the endpoints. Moreover, for every $0\leq \tau\leq \epsilon$,
 \begin{eqnarray*}
 	\lvert \phi(\tau)-\widetilde{\phi}(\tau)\rvert &=&\lvert \phi(\epsilon)+\int_\epsilon^\tau \phi_\tau\,d\sigma-\widetilde{\phi}(\epsilon)-\int_\epsilon^\tau \widetilde{\phi}_\tau\,d\sigma\rvert\\
 	&\leq&\lvert \phi(\epsilon)-\widetilde{\phi}(\epsilon)\rvert+\epsilon\lVert \phi_\tau-\widetilde{\phi}_\tau\rVert_{\mathcal{L}^\infty(0,\epsilon)}.
 \end{eqnarray*}
 Hence, $\phi\to\widetilde{\phi}$ uniformly on $[0,\epsilon]$. Together with the uniform convergence of $\phi_\tau$ and the convergence of $\phi_{\tau\tau}$ in $\mathcal{L}^p(0,\epsilon)$, this gives $\phi\to\widetilde{\phi}$ in $\mathcal{W}^{2,p}(0,\epsilon)$, for every $1\leq p<\infty$.
 
 On the limiting branch $s>0$, we have $\tau={\rm sech}\,s$ and, hence, $\tanh s=\sqrt{1-\tau^2}$. Therefore, after straightforward simplifications
 $$\gamma_{\mu,c}(\tau)\longrightarrow \frac{e^{\widetilde{\phi}(\tau)}}{\sqrt{1+\tau^2}}\left(1,0,\tau\right)=\frac{1}{\sqrt{2}-\sqrt{1-\tau^2}}\left(1,0,\tau\right),$$
 in $\mathcal{W}^{2,p}(0,\epsilon)$, for every $1\leq p<\infty$. The branch $s<0$ is treated similarly, but now $\tanh s=-\sqrt{1-\tau^2}$, and so
 $$\gamma_{\mu,c}(\tau)\longrightarrow\frac{1}{\sqrt{2}+\sqrt{1-\tau^2}}\left(1,0,\tau\right).$$
 These are the two endpoint parameterizations of the limiting semicircle. Together with the smooth convergence away from the endpoints, this proves that the generalized elastic curves $\gamma_{\mu,c}$ converge in $\mathcal{W}^{2,p}$, for every $1\leq p<\infty$, up to their endpoints, to the upper semicircle of radius one centered at $(\sqrt{2},0,0)$.
\end{proof}

The convergence established in Lemmas \ref{lemma1} and \ref{lemma2} is illustrated in Figure \ref{F3}, which shows the generating curves corresponding to the rotational tori in Figure \ref{F0}. In the Poincar\'e disk model, this figure also illustrates how the generalized elastic curves evolve toward a hyperbolic geodesic, represented by an arc of a Euclidean circle meeting the ideal boundary orthogonally.

\begin{figure}[h!]
	\centering
	\includegraphics[height=5.58cm]{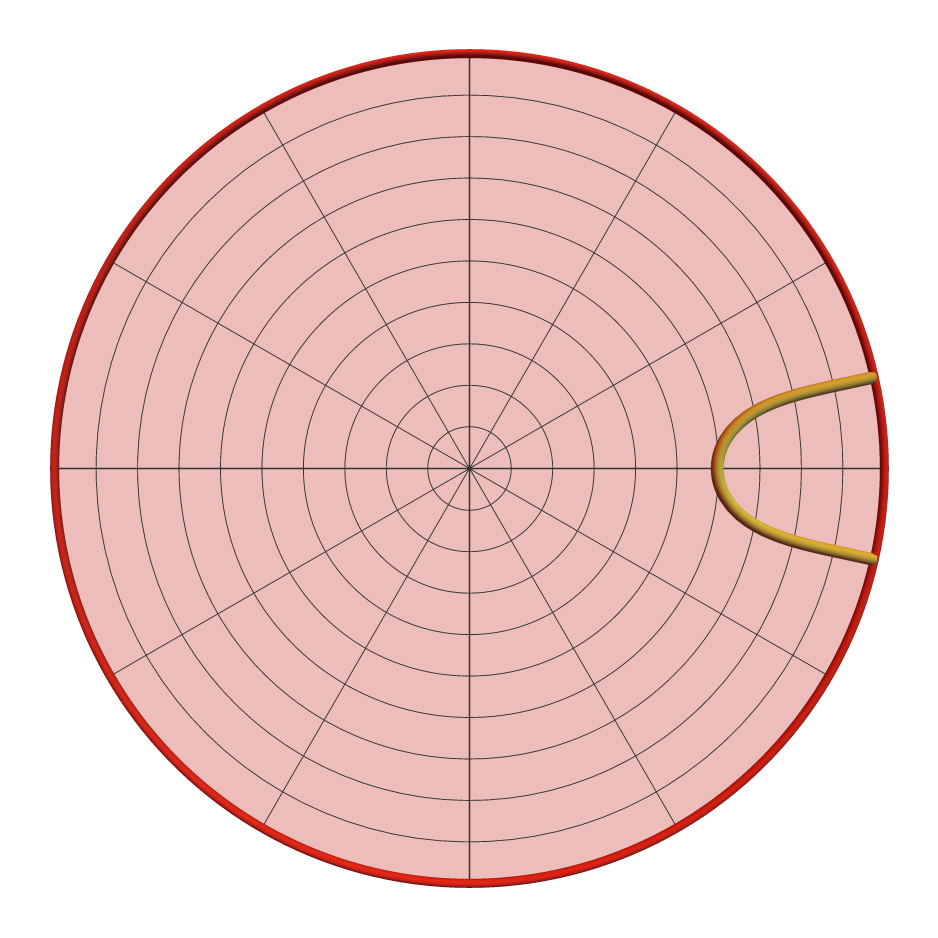}
	\includegraphics[height=5.58cm]{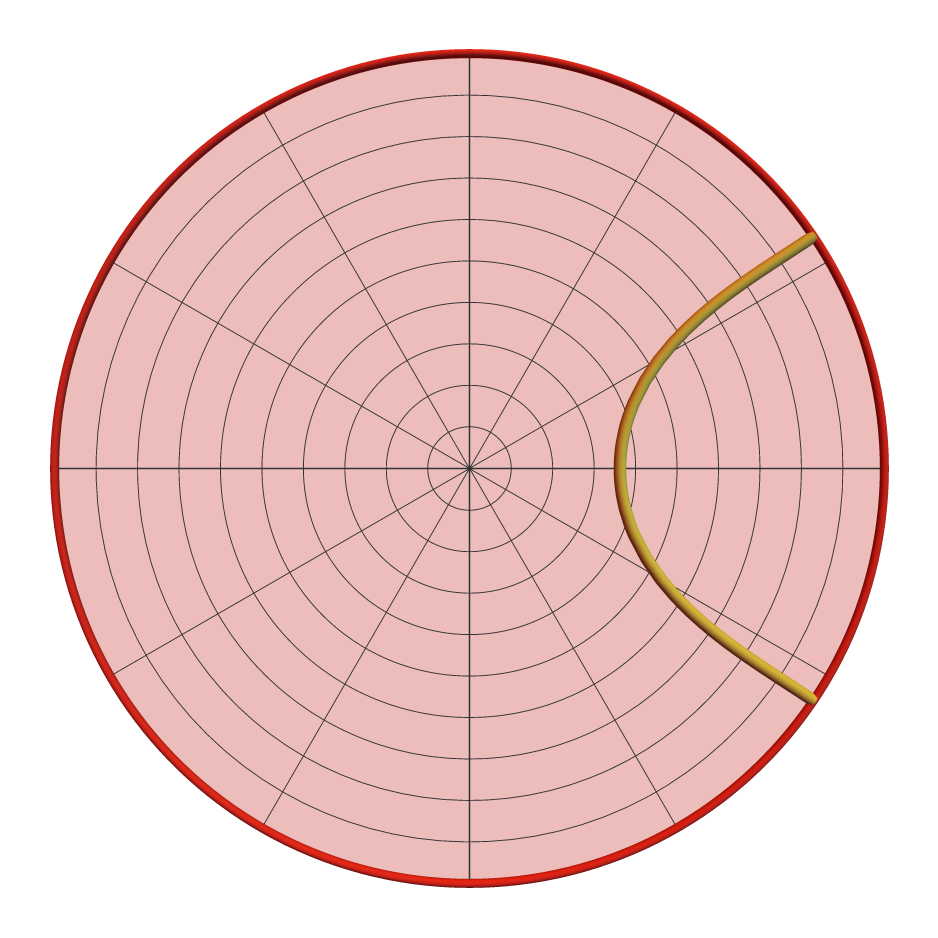}
	\includegraphics[height=5.58cm]{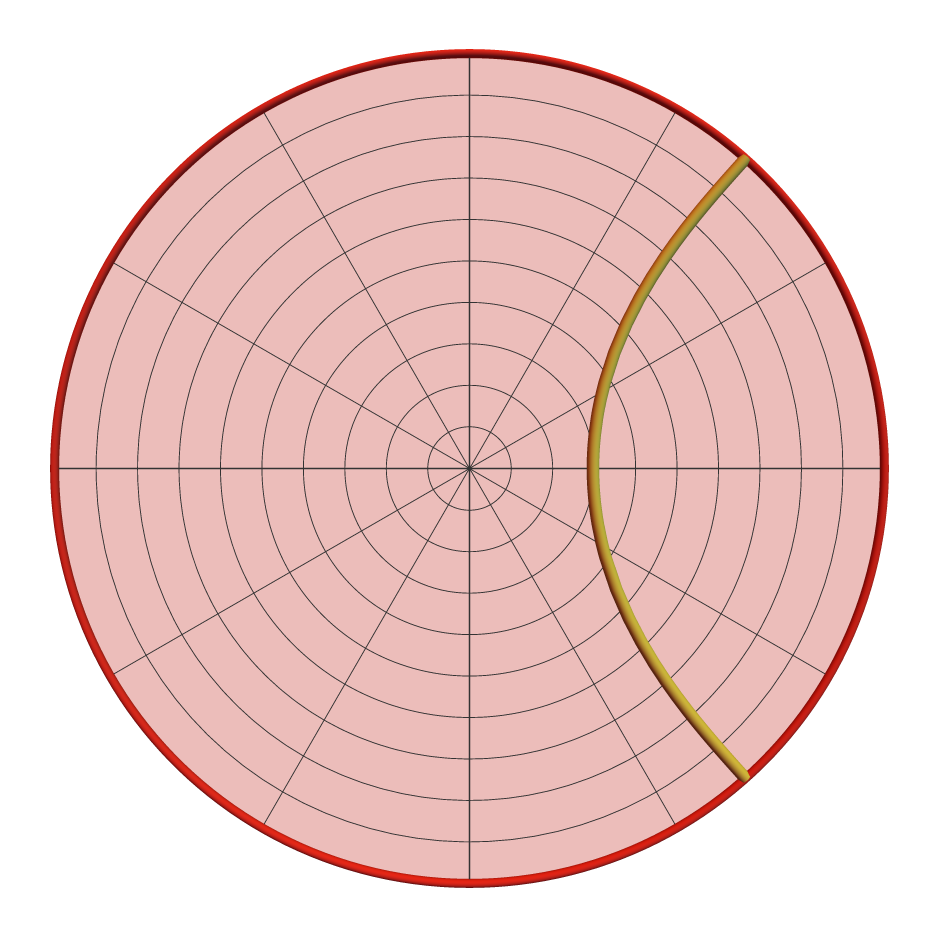}
	\caption{\small{Generalized elastic curves $\gamma_{\mu,c}$ in the Poincar\'e disk model of $\mathbb{H}^2$, with $c=-\mu/\sqrt{2}$ and, from left to right, $\mu=-0.6$, $\mu=-2$, and $\mu=-6$. These curves generate the rotational tori displayed in Figure \ref{F0} and converge, as $\mu\to-\infty$, to the limiting curve $\widetilde{\gamma}$ described in Lemmas \ref{lemma1} and \ref{lemma2}.}}
	\label{F3}
\end{figure}

The preceding lemmas show that, when viewed as Euclidean curves in the closed upper half-plane, the generalized elastic curves converge smoothly away from their endpoints and in $\mathcal{W}^{2,p}$ up to their endpoints for every $1\leq p<\infty$. This convergence cannot be strengthened to $\mathcal{W}^{2,\infty}$, as can be seen from the behavior of their Euclidean curvatures at the endpoints. The next lemma makes this obstruction precise.

\begin{lemma}\label{lemma3}
	Let $\mu<0$ and $c=-\mu/\sqrt{2}$. As $\mu\to-\infty$, the Euclidean curvatures of the generalized elastic curves $\gamma_{\mu,c}$ at their endpoints converge to $\sqrt{2}-1$ and $\sqrt{2}+1$, respectively, whereas the limiting semicircle has Euclidean curvature one. Consequently, the curves $\gamma_{\mu,c}$ do not converge in $\mathcal{W}^{2,\infty}$ up to their endpoints.
\end{lemma}
\begin{proof}
	Recall that, near either endpoint, the curves $\gamma_{\mu,c}$ can be parameterized as \eqref{gammatau}. From the formulas obtained in the proof of Lemma \ref{lemma2}, we have $x(0)=0$, $x_\tau(0)=0$, and so $\phi_\tau(0)=0$ and $\phi_{\tau\tau}(0)=0$.
	
	From the standard formula for the Euclidean curvature of a plane curve,
	$$\overline{\kappa}(0)=e^{-\phi(0)}\lvert 1-\phi_{\tau\tau}(0)\rvert=e^{-\phi(0)}.$$
	On the branch $s>0$, Lemma \ref{lemma2} gives $e^{\phi(0)}\to\sqrt{2}+1$. Hence, $\overline{\kappa}(0)\to\sqrt{2}-1$. On the branch $s<0$, we similarly have $e^{\phi(0)}\to\sqrt{2}-1$ and therefore $\overline{\kappa}(0)\to\sqrt{2}+1$.
	
	The limiting semicircle has radius one and thus Euclidean curvature one at both endpoints. Therefore, the endpoint curvatures of $\gamma_{\mu,c}$ do not converge to those of the limiting semicircle. If the curves converged in $\mathcal{W}^{2,\infty}$ up to their endpoints, then, since the parameterizations remain regular there, their Euclidean curvatures would converge in $\mathcal{L}^\infty$. As these curvatures are continuous up to the endpoints, the endpoint mismatch rules this out. Hence, the curves $\gamma_{\mu,c}$ do not converge in $\mathcal{W}^{2,\infty}$ up to their endpoints.
\end{proof}

\bigskip

\begin{flushleft}
	\'Alvaro P{\footnotesize \'AMPANO}\\
	Department of Mathematics and Statistics, Texas Tech University, Lubbock, TX, 79409, USA\\
	E-mail: alvaro.pampano@ttu.edu
\end{flushleft}

\begin{flushleft}
	Aaron J. T{\footnotesize YRRELL}\\
	Department of Mathematics, University of Notre Dame, Notre Dame, IN 46556, USA\\
	E-mail: atyrrell@nd.edu
\end{flushleft}

\end{document}